\documentclass[reqno]{amsart}
\usepackage[T1]{fontenc}
\usepackage{cite,lmodern}
\usepackage{microtype}
\usepackage{amssymb,mathtools}
\usepackage{xcolor}
\definecolor{citationblue}{RGB}{0,70,140}
\usepackage[
  colorlinks=true,
  linkcolor=blue,
  citecolor=citationblue,
  urlcolor=citationblue
]{hyperref}
\usepackage{enumitem}
\usepackage{array}
\usepackage{tikz}
\newtheorem{theorem}{Theorem}[section]
\newtheorem{proposition}[theorem]{Proposition}
\newtheorem{lemma}[theorem]{Lemma}
\newtheorem{corollary}[theorem]{Corollary}
\newtheorem{remark}[theorem]{Remark}
\newcommand{\Sp}[1]{\mathcal S^{#1}}
\newcommand{\K}{\mathcal K(\mathcal H)}
\newcommand{\B}{\mathcal B(\mathcal H)}
\newcommand{\dd}{\Psi_{f,\lambda}}
\newcommand{\Tr}{\operatorname{Tr}}

\title[Arazy conjecture concerning Schur multipliers]
 {Arazy's conjecture  concerning Schur multipliers: revisited and resolved}

\subjclass[2020]{Primary 47B10; Secondary 47A55, 47A60}
\keywords{Schur--Hadamard multiplier, Schatten ideal, divided
difference
}

\begin{document}

\begin{abstract}
Let $\Sp r$ denote the Schatten--von Neumann class and let
$S_{\Psi_{f,\lambda}}$ be the Schur--Hadamard multiplier whose
symbol is the divided-difference matrix of $f$ along $\lambda$.
Let $0<\alpha,r<\infty$, let $f\in C^1([-1,1])$ satisfy $f(0)=0$, 
$|f'(t)|\lesssim |t|^\alpha$, and let $\lambda\in\ell^r$ be real
with $\|\lambda\|_{\ell^\infty}\leq1$.  We
determine the pairs $0<p,q\leq\infty$ for which
$S_{\Psi_{f,\lambda}}:\Sp q\to\Sp p$ is bounded for every such
$f$ and $\lambda$.  
Our principal new positive estimates treat
$q=\infty,1$ and $0<p<1$. 
Precisely, we show that  the boundedness holds exactly when
\[
 \frac1p\leq\frac{\alpha}{r}
       +\min\!\left\{1,\frac1q\right\}.
\]
In particular, this resolves the untreated cases in [Arazy, PAMS, 1982] and [Potapov, Sukochev, Tomskova, Adv. Math., 2015].  
Our  method also delivers a new proof of the main results in just cited papers. 
\end{abstract}

\author[J. Huang]{Jinghao Huang}
\address{Institute for  Advanced Study in  Mathematics of HIT, Harbin Institute of Technology, Harbin, 150001, China}
\email{{\color{blue}jinghao.huang@hit.edu.cn}}
 
 \author[F. Sukochev]{Fedor Sukochev}
\address{School of Mathematics and Statistics, University of NSW, Sydney,  2052, Australia}
\email{\color{blue}f.sukochev@unsw.edu.au}
\thanks{J. Huang were supported the NNSF of China (No. 12301160,  12471134 and 12671159); F. Sukochev was supported by the ARC (DP230100434)}

\maketitle

\section{Introduction}
\subsection{Notations}
Throughout, $\mathcal H$ denotes a separable infinite-dimensional
complex Hilbert space.  Its inner product
$\langle\cdot,\cdot\rangle$ is taken to be linear in the first
variable.  The symbols
$\mathcal B(\mathcal H)$ and $\mathcal K(\mathcal H)$ denote,
respectively, the bounded and compact operators on $\mathcal H$.
For $A\in\mathcal K(\mathcal H)$, let
$|A|=(A^*A)^{1/2}$ and let
\[
 s_1(A)\geq s_2(A)\geq\cdots\geq0
\]
be the singular values of $A$, repeated according to
multiplicity.  We write $\Tr$ for the canonical
trace.  For $0<p<\infty$, the Schatten--von Neumann class is
\[
 \Sp p
 =
 \left\{A\in\mathcal K(\mathcal H):
   \|A\|_p
   :=
   \bigl(\Tr|A|^p\bigr)^{1/p}
   =
  \left(\sum_{j\geq1}s_j(A)^p\right)^{1/p}
    <\infty\right\}.
\]
For standard background on Schatten ideals we refer to Gohberg and
Kre\u{\i}n \cite{GohbergKrein} and Simon \cite{Simon} (see also \cite{LSZ}).
It is a Banach space for $p\geq1$ and a  quasi-Banach
space for $0<p<1$.  For the exponent $p=\infty$ we set
\[
 \Sp\infty=\mathcal K(\mathcal H),
 \qquad
 \|A\|_\infty=\|A\|_{\mathcal B(\mathcal H)}.
\]
This follows Arazy's convention: his $C_p$ is our $\Sp p$ for
$1\leq p<\infty$, while
\[
 C_\infty=\Sp\infty=\mathcal K(\mathcal H)
 \subsetneq\mathcal B(\mathcal H).
\]
The weak trace class used below is (see \cite{Simon} or \cite[Chapter I]{LSZ})
\[
 \Sp{1,\infty}
 =
 \left\{A\in\mathcal K(\mathcal H):
   \|A\|_{1,\infty}:=
   \sup_{j\geq1}j\,s_j(A)<\infty\right\}.
\]
For $0<p<\infty$, $\ell^p$ is the space of scalar sequences
$x=(x_j)_{j\geq1}$ with
\[
 \|x\|_{\ell^p}
 =
 \left(\sum_{j\geq1}|x_j|^p\right)^{1/p}<\infty,
\]
and $\ell^\infty$ carries the supremum norm.
We abbreviate
$\|x\|_p:=\|x\|_{\ell^p}$.

Fix an orthonormal basis $(e_j)_{j\geq1}$ of $\mathcal H$, and
let $e_{jk}$ be the associated matrix units, determined by
$e_{jk}e_\ell=\delta_{k\ell}e_j$, where $\delta_{k\ell}$ is the
Kronecker delta.  If
$A\in\mathcal B(\mathcal H)$, write
$a_{jk}=\langle Ae_k,e_j\rangle$ for its matrix entries.  Given a scalar
matrix $M=(m_{jk})_{j,k\geq1}$, its Schur--Hadamard action is
defined initially on finite matrices by
\[
 S_M(A)=(m_{jk}a_{jk})_{j,k\geq1}.
\]
For $0<p,q\leq\infty$, we call $M$ a Schur multiplier from
$\Sp q$ to $\Sp p$ if this map extends boundedly from $\Sp q$ to
$\Sp p$.  Since finite matrices are operator-norm dense in
$\Sp\infty=\mathcal K(\mathcal H)$, this convention is
unambiguous also when $q=\infty$.  We denote the operator
(quasi-)norm by
\[
 \|S_M:\Sp q\to\Sp p\|
 =
 \sup_{\|A\|_q\leq1}\|S_M(A)\|_p.
\]
Here and below, a subscript on a
constant records the parameters on which it may depend, and
$X\lesssim_\gamma Y$ means
$X\leq C_\gamma Y$ for a constant $C_\gamma>0$.

\subsection{Arazy's conjecture}
For a bounded real sequence
$\lambda=(\lambda_j)_{j\geq1}$ with
$\|\lambda\|_{\ell^\infty}\leq1$, let $H_\lambda$ be the diagonal
self-adjoint operator defined by
$H_\lambda e_j=\lambda_j e_j$.  Thus, for
$0<r<\infty$, $\lambda\in\ell^r$ exactly when
$H_\lambda\in\Sp r$, and then
$\|H_\lambda\|_r=\|\lambda\|_{\ell^r}$.  Given such a sequence
and a function $f\in C^1([-1,1])$, Arazy considered
the divided-difference matrix
\begin{align}
 \Psi_{f,\lambda}(j,k)
 =
 \begin{cases}
 \displaystyle
 \frac{f(\lambda_j)-f(\lambda_k)}
      {\lambda_j-\lambda_k},
 &\lambda_j\ne\lambda_k, \\[1.2ex]
 f'(\lambda_j),&\lambda_j=\lambda_k.\label{def Psi}
 \end{cases}
\end{align}
The corresponding Schur--Hadamard multiplier is therefore
\[
 S_{\Psi_{f,\lambda}}(A)
   =\bigl(\Psi_{f,\lambda}(j,k)a_{jk}\bigr)_{j,k}.
\]
Such matrices are the discrete form of double operator integrals
(DOIs) with divided-difference kernel.
Birman and Solomyak
developed the systematic theory of DOIs in their foundational
papers
\cite{BirmanSolomyak1965,BirmanSolomyakMultipliers}.  In particular,
\cite[Theorem~5]{BirmanSolomyak1965} (as recalled in
\cite[Theorem~1.1]{Arazy}) implies that if $f'$ is H\"older
continuous on $[-1,1]$, then
\[
 S_{\Psi_{f,\lambda}}:\B\longrightarrow\B
 \quad\text{and}\quad
 S_{\Psi_{f,\lambda}}:\Sp p\longrightarrow\Sp p,
 \qquad 1\leq p\leq\infty,
\]
for every bounded real sequence $\lambda$.  For a later account of
the DOI theory, see their survey \cite{BirmanSolomyak}.

In 1982 Arazy proved \cite[Theorem 2.5]{Arazy} that, for
$0<\alpha,r<\infty$, if
$f(0)=0$, $|f'(t)|\leq C|t|^\alpha$, and
$\lambda\in\ell^r$, then
\[
 S_{\Psi_{f,\lambda}}:\Sp{p_1}\longrightarrow\Sp{p_2}
\]
is bounded when
\begin{equation}\label{eq:arazy-relation}
 1\leq p_2\leq2\leq p_1\leq\infty,
 \qquad
 \frac1{p_2}=\frac1{p_1}+\frac{\alpha}{r}.
\end{equation}
  Arazy then conjectured (see also \cite{PST,CGPT}) that the above result remains valid when the two
indices lie on the same side of $2$, i.e., 
\[
 1\le p_2\le p_1\le 2,
 \qquad\text{or}\qquad
 2\le p_2\le p_1\le \infty
\]
under the exponent relation in \eqref{eq:arazy-relation}.
Later, Arazy and Friedman\cite[Theorem 3.6]{ArazyFriedman}, Potapov and Sukochev\cite[Theorem 13]{PotapovSukochev2014} verified the conjecture for some special functions.


\subsection{Main result}

In \cite{PST}, 
Potapov, Sukochev, and Tomskova  considered a   question which is more general than Arazy conjecture:
\begin{quote}
Under which conditions on $p$, $q$, on the function $f$ and on the sequence $\lambda$, 
    is the Schur multiplier   $S_{\Psi_{f,\lambda}}:\Sp q\to \Sp p$ bounded?\cite[Question A]{PST}
\end{quote}
It was shown in 
\cite[Corollary~7]{PST}  that
\[
 S_{\Psi_{f,\lambda}}:\Sp q\longrightarrow\Sp p
 \quad\text{is bounded whenever}\quad
 0<p<\infty,\quad 1<q<\infty,\quad
 \frac1p=\frac{\alpha}{r}+\frac1q.
\]
This together with the multiplier duality recorded in
\cite[pp.~63--64]{Arazy} (see Lemma \ref{lem:duality} below) yields
the case when  $1<p<\infty$ and $q=\infty $.
In summary,  
the cases when
\[
 q\le 1,\qquad 0 <p<\infty,\qquad  \mbox{ and }\qquad 
 q=\infty,\qquad p\le 1 
\]
remain untreated (see \cite[Section~4.2]{PST} for a discussion concerning
a partial answer to 
the case when $q=\infty $ via  an interpolation method).

The main purpose of the present paper is to treat the Arazy conjecture and \cite[Question A]{PST} 
for the full range of indices $ 0<p,q\leq \infty$, working with quasi-Banach Schatten ideals.
The main new tool used for this case is the  weak $(1,1)$-estimate established in \cite{CPSZ}. 
In Remark \ref{rem:q-greater-than-one}, we indicate that our method allows to completely cover the main results in \cite{Arazy} and \cite{PST}. 

\begin{theorem}\label{thm:main-classification}
Fix $0<\alpha,r<\infty$ and $0<p,q\leq\infty$, with
$1/\infty=0$.  The following are equivalent.
\begin{enumerate}[label=\textup{(\roman*)}]
\item For every $f\in C^1([-1,1])$  satisfying
$f(0)=0$ and
\[
 |f'(t)|\leq C_f|t|^\alpha,\qquad -1\leq t\leq1,
\]
for some constant $C_f>0$, and every real sequence
$\lambda\in\ell^r$ satisfying $\|\lambda\|_\infty\leq1$, the map
\[
 S_{\Psi_{f,\lambda}}:\Sp q\longrightarrow\Sp p
\]
is bounded.
\item
\begin{equation}\label{eq:full-condition}
 \frac1p\leq \frac{\alpha}{r}
       +\min\!\left\{1,\frac1q\right\}.
\end{equation}
\end{enumerate}
\end{theorem}




Recently, 
Conde-Alonso, Gonz\'alez-P\'erez, Parcet, and Tablate \cite{CGPT}
proved H\"ormander--Mikhlin criteria for Schur
multipliers; see also \cite{GPPR,PotapovSukochev2011,OS} for related results. 
Their works can be viewed as addressing the limiting case of Arazy conjecture when $\alpha =  0$ and $p=q$\cite{Parcet}.
The weak 
$(1,1)$-estimate from \cite{CPSZ} also addresses the case when $\alpha =0$ and, in this sense,  there is a connection between just cited papers and our current work. However, the results in the present paper do not intersect with theirs because 
they do not address the original
Arazy conjecture with unequal-exponent ($p\ne q$).

In the last section of this paper, we briefly describe how our results and techniques can be extended to $L_p$-spaces affiliated with general semifinite von Neumann algebras $\mathcal{M}$ and explain the difference when $\mathcal{M}$ is atomless.

\subsection{Sharpness  of the exponent condition}

Arazy proved the necessary exponent condition for Banach indices
in \cite[Proposition~3.1]{Arazy}.  For $0<p<1$, Aleksandrov and
Peller \cite[Theorem~4.1]{AleksandrovPeller} proved that the
diagonal matrix $(d_j\delta_{jk})_{j,k\geq1}$ defines a bounded
Schur multiplier on $\Sp p$ exactly when
$d\in\ell^{p/(1-p)}$ (see also \cite[Lemma 2.2.3]{McS}).


In Section  \ref{sec:sharp-classification} below, we 
 prove that the condition \eqref{eq:full-condition} is sharp,  in other words, that 
 the Schur multipliers in 
 Arazy conjecture and \cite[Question A]{PST} are unbounded when
\[
\begin{array}{ll}
q=\infty:
 & \displaystyle \frac1p>\frac{\alpha}{r},\\[2ex]
1<q<\infty:
 & \displaystyle \frac1p>\frac{\alpha}{r}+\frac1q,\\[2ex]
0<q\leq1:
 & \displaystyle \frac1p>1+\frac{\alpha}{r}
\end{array}
\]
(see Theorem~\ref{thm:sharp-failure}  below). 
Writing $q$ for the input exponent and $p$ for the target
exponent, our main result is the following sharp classification.
In the table below, the word ``holds'' means
 that 
 the corresponding Schur multiplier symbol is bounded from $\Sp q$ to $\Sp p$ 
 for every
$f\in C^1([-1,1])$ satisfying
$f(0)=0$ and
$|f'(t)|\leq C_f|t|^\alpha$ on $[-1,1]$ for some $C_f>0$,
and every real sequence
$\lambda\in\ell^r$ satisfying $\|\lambda\|_{\ell^\infty}\leq1$.
\begin{center}
\begingroup
\small
\renewcommand{\arraystretch}{1.18}
\begin{tabular}{@{}
 >{\raggedright\arraybackslash}p{0.2\textwidth}
 >{\raggedright\arraybackslash}p{0.35\textwidth}
 >{\raggedright\arraybackslash}p{0.41\textwidth}@{}}
\hline
\textbf{Input range}
& \textbf{Conclusion}
& \textbf{Source}\\
\hline
$1<q<\infty$
 & For $0<p<\infty$, holds iff
   $1/p\leq\alpha/r+1/q$.
 & \cite[Corollary~7]{PST};
   Corollary~\ref{cor:finite-positive};
   Theorem~\ref{thm:sharp-failure}.\\
$q=\infty$
 & For $0<p\leq\infty$, holds iff
   $1/p\leq\alpha/r$.
 & \cite[Theorem~2.5]{Arazy};
   \cite[Corollary~7]{PST} and
   \cite[pp.~63--64]{Arazy};
   Corollary~\ref{cor:qinf-positive-range};
   Theorem~\ref{thm:sharp-failure}.\\
$0<q\leq1$
 & Holds iff $1/p\leq1+\alpha/r$.
 & Corollary~\ref{cor:finite-positive};
   Theorem~\ref{thm:sharp-failure}.\\
\hline
\end{tabular}
\endgroup
\end{center}

\section{The boundedness of $S_{\Psi_{f,\lambda }}: \Sp \infty  \to \Sp p$}\label{s 2}
Throughout the paper, 
for $0<\alpha<\infty$
and functions   $f\in C^1([-1,1])$ with $f(0)=0$ satisfying  $ \sup_{0<|t|\leq1}\frac{|f'(t)|}{|t|^\alpha}<\infty $, 
we use the following notation 
\[
 L_\alpha(f)
 :=
 \sup_{0<|t|\leq1}\frac{|f'(t)|}{|t|^\alpha}<\infty .
\]

  Adding a constant to $f$ does not change
its divided differences, so the normalization $f(0)=0$ is
harmless.  Let $\lambda$ be a real sequence with
$\|\lambda\|_\infty\leq1$.  The matrix $\dd$ defined above is
symmetric:
\begin{equation}\label{eq:symmetric}
 \dd(j,k)=\dd(k,j).
\end{equation}
We use the following result.
\begin{theorem}\cite[Corollary~7]{PST} \label{thm:PST}
Let $0<\alpha,r<\infty$.  Suppose that $f\in C^1([-1,1])$ is real-valued,
$f(0)=0$, and $L_\alpha(f)<\infty$, and that
$\lambda\in\ell^r$ is real with $\|\lambda\|_\infty\leq1$.  If
\[
 0<p<\infty,\qquad 1<q<\infty,\qquad
 \frac1p=\frac{\alpha}{r}+\frac1q,
\]
then
\begin{equation}\label{eq:PST}
 \|S_{\dd}:\Sp q\longrightarrow\Sp p\|
 \leq C_{\alpha,p,q,f}\,\|\lambda\|_r^\alpha,
\end{equation}
where the constant $C_{\alpha,p,q,f}$ is independent of $\lambda$.
\end{theorem}




Arazy recorded the following multiplier-duality principle in his
concluding remark \cite[pp.~63--64]{Arazy}, citing the standard
duality facts but not giving a   proof.   Below, we include a full proof for completeness. 
For
$1\leq s\leq\infty$, put
\[
 \mathcal C_s=
 \begin{cases}
  \Sp s,&1\leq s<\infty,\\
  \K,&s=\infty,
 \end{cases}
 \qquad
 \mathcal C_s^\times=
 \begin{cases}
  \B,&s=1,\\
  \Sp{s'},&1<s<\infty,\\
  \Sp1,&s=\infty,
 \end{cases}
\]
where $1/s+1/s'=1$.  
The dual  
 $\mathcal C_s^*$ of $\mathcal C_s$ can be identified with 
 $\mathcal C_s^\times$ via trace duality.

\begin{lemma}\label{lem:duality}
Let $1\leq p,q\leq\infty$, let $M=(m_{jk})_{j,k\geq1}$ be a scalar
matrix, and let 
$M^{\mathsf T}=(m_{kj})_{j,k\geq1}$.  If
\[
 S_M:\mathcal C_q\longrightarrow\mathcal C_p
\]
is bounded, then its Banach adjoint is
\[
 S_{M^{\mathsf T}}:
 \mathcal C_p^\times\longrightarrow\mathcal C_q^\times.
\]
Moreover,
\[
 \|S_{M^{\mathsf T}}:
   \mathcal C_p^\times\to\mathcal C_q^\times\|
 =
 \|S_M:\mathcal C_q\to\mathcal C_p\|.
 \]
\end{lemma}

\begin{proof}
Let
$T^*:\mathcal C_p^\times\to\mathcal C_q^\times$
be the Banach adjoint of $T=S_M$ under the pairing
$(A,B)\mapsto\Tr(AB)$.  For
$A\in\mathcal C_p^\times$, put $C=T^* (A)$.
Write $c_{jk}=\langle Ce_k,e_j\rangle$ for the matrix entries of
$C$.
For each matrix unit $e_{kj}\in\mathcal C_q$, we have 
\[
 c_{jk}
 =\Tr(Ce_{kj})
 =\Tr\!\bigl(A\,T(e_{kj})\bigr)
 =m_{kj}a_{jk}.
\]
Thus, $T^* =S_{M^{\mathsf T}}$.  
The 
equality of the
norms is a well-known fact
\cite[Theorem~4.10, p.~98]{Rudin}.
\end{proof}

\begin{corollary} 
\label{cor:duality-compact}
Let $1<p<\infty$, let $p'=p/(p-1)$, and let $M$ be a scalar
matrix.  If
\[
 S_M:\Sp{p'}\longrightarrow\Sp1
\]
is bounded, then
\[
 S_{M^{\mathsf T}}:\B\longrightarrow\Sp p
\]
is bounded with the same norm, which is the same as $ \left\|S_{M^{\mathsf T}}\mid_{ \mathcal{K}(\mathcal{H})} : \mathcal{K}(\mathcal{H})  \to \mathcal{S}^p\right\|  $.  
\end{corollary}

\begin{proof}One only needs to observe that 
\[
 \|S_{M^{\mathsf T}}:\K\to\Sp p\| \stackrel{\tiny \mbox{Lemma~\ref{lem:duality}}}{=} \|S_M:\Sp{p'}\to\Sp1\|
 \stackrel{\tiny \mbox{Lemma~\ref{lem:duality}}}{=}
 \|S_{M^{\mathsf T}}:\B\to\Sp p\|
.
\]
\end{proof}

The following is an immediate consequence of  the symmetry condition 
\eqref{eq:symmetric},  Theorem  \ref{thm:PST} and Corollary \ref{cor:duality-compact}, which resolves the untreated case in \cite[Section 4.2]{PST} for $1< p<\infty$.
\begin{proposition}
\label{prop:qinf-banach}
Let $1<p<\infty$, 
$0<\alpha<\infty$ and suppose
$\lambda\in\ell^{\alpha p}$ with
$\|\lambda\|_\infty\leq1$.  Then
\begin{equation}\label{eq:qinf-banach}
 \|S_{\dd}:\B\longrightarrow\Sp p\|
 \leq
 C_{\alpha,p,f}\,
 \|\lambda\|_{\alpha p}^{\alpha}.
\end{equation}
Moreover, the same estimate holds after restricting the
domain to
\(
\Sp\infty=\K.
\)
\end{proposition}
\begin{proof}
Since \( 
 1 
  =\frac{\alpha}{\alpha p}+\frac1{p'},
\)
 it follows from   Theorem \ref{thm:PST} (with
\(
 ( \mathcal{S}^1,\mathcal{S}^{p'},\ell^{\alpha p})
\))   that 
\[
 \|S_{\dd}:\Sp{p'}\longrightarrow\Sp1\|
 \leq
 C_{\alpha,p,f}\,\|\lambda\|_{\alpha p}^{\alpha}.
\]
Now, combining  Corollary~\ref{cor:duality-compact} and condition 
\eqref{eq:symmetric}, we  complete the proof.
\end{proof}

For $p=1$, the boundedness of the Schur multiplier
$S_{\dd}:\Sp\infty\to\Sp1$ under
$\lambda\in\ell^\alpha$, $0<\alpha<\infty$, is already contained in
\cite[Theorem~2.5]{Arazy}.
Below, we prove the case when $0<p<1$,  which were not treated  in \cite[Section 4.2]{PST} and \cite{Arazy}.

\begin{theorem} 
\label{thm:qinf-quasi}
Let $0<p\leq1$, $0<\alpha<\infty$ and suppose
$\lambda\in\ell^{\alpha p}$ with
$\|\lambda\|_\infty\leq1$.  Then
\begin{equation}\label{eq:qinf-quasi}
 \|S_{\dd}:\B\longrightarrow\Sp p\|
 \leq
 2^{1/p}L_\alpha(f)\,
 \|\lambda\|_{\alpha p}^{\alpha}.
\end{equation}
\end{theorem}

\begin{proof}
For $x\ne y$, we have 
\[
 \frac{f(x)-f(y)}{x-y}
 =\int_0^1 f'\bigl(y+s(x-y)\bigr)\,ds.
\]
Since
$|y+s(x-y)|\leq\max\{|x|,|y|\}$ for $0\leq s\leq1$,
it follows that \[
 \left| \frac{f(x)-f(y)}{x-y}\right|
 \le   L_\alpha(f)
 \max\{|x|^\alpha ,|y|^\alpha\},
\]
which together with 
 the definition of $\dd(j,k)$ when $\lambda_j=\lambda_k$
 (see \eqref{def Psi})
 yields that 
\begin{equation}\label{eq:divided-difference-pointwise}
 |\dd(j,k)|
 \leq L_\alpha(f)
       \max\{|\lambda_j|^\alpha,|\lambda_k|^\alpha\}.
\end{equation}
For $A=(a_{jk})\in\B$, decompose its Schur product as
\[
 S_{\dd}(A)=\sum_{j\geq1}R_j+\sum_{k\geq1}C_k,
\]
where
\[
 R_j=\sum_{\{k:\,|\lambda_k|\leq|\lambda_j|\}}
       \dd(j,k)a_{jk}e_{jk},
 \qquad
 C_k=\sum_{\{j:\,|\lambda_j|<|\lambda_k|\}}
       \dd(j,k)a_{jk}e_{jk}.
\]

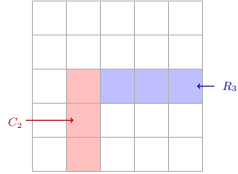
\begin{figure}[ht]
\makeatletter
\renewcommand{\@captionfont}{\normalfont\tiny}
\makeatother
\setlength{\abovecaptionskip}{-3pt}
\centering
\scalebox{0.5}{%
\begin{tikzpicture}[
  x=0.9cm,
  y=-0.9cm,
  every node/.style={font=\small}
]

  \fill[red!25] (1,2) rectangle (2,5);

  \fill[blue!25] (2,2) rectangle (5,3);

  \foreach \i in {0,...,5}{
    \draw[gray!60] (\i,0)--(\i,5);
    \draw[gray!60] (0,\i)--(5,\i);
  }


  \node[blue!65!black,anchor=west] at (5.45,2.5) {$R_3$};
  \draw[->,blue!65!black] (5.35,2.5)--(4.85,2.5);

  \node[red!65!black,anchor=north] at (-0.5,3.3) {$C_2$};
  \draw[->,red!65!black] (-0.2,3.5)--(1.2,3.5);
\end{tikzpicture}
}
\caption{
$R_3$ (blue) and $C_2$ (red)
when
$|\lambda_1|>\cdots>|\lambda_5|>\cdots$}
\label{fig:Rj-Ck}
\end{figure}
In particular, their ranks are at most one. 
Therefore, 
\[
\begin{aligned}
 \|R_j\|_p
 =\|R_j\|_2
  &=\left(\sum_{\{k:\,|\lambda_k|\leq|\lambda_j|\}}
      |\dd(j,k)a_{jk}|^2\right)^{1/2}\\
 &\stackrel{\eqref{eq:divided-difference-pointwise}}{\leq} L_\alpha(f)|\lambda_j|^\alpha
  \left(\sum_k|a_{jk}|^2\right)^{1/2}\leq L_\alpha(f)|\lambda_j|^\alpha\|A\|_\infty.
\end{aligned}
\]
Similarly,
\[
\begin{aligned}
 \|C_k\|_p
 =\|C_k\|_2
 & =\left(\sum_{\{j:\,|\lambda_j|<|\lambda_k|\}}
      |\dd(j,k)a_{jk}|^2\right)^{1/2}\\
&\stackrel{\eqref{eq:divided-difference-pointwise}}{\leq}  L_\alpha(f)|\lambda_k|^\alpha
  \left(\sum_j|a_{jk}|^2\right)^{1/2}\leq L_\alpha(f)|\lambda_k|^\alpha\|A\|_\infty.
\end{aligned}
\]
Using the $p$-triangle inequality in $\Sp p$
\cite[Proposition~4.6(i)]{FackKosaki}, we obtain
\[
 \|S_{\dd}(A)\|_p^p
 \leq
 2L_\alpha(f)^p
 \|\lambda\|_{\alpha p}^{\alpha p}\|A\|_\infty^p.
\]
This   proves \eqref{eq:qinf-quasi}.
\end{proof}
Using  Proposition~\ref{prop:qinf-banach} when $1<p<\infty$ and
Theorem~\ref{thm:qinf-quasi} when $0<p\leq1$, we obtain the following corollary.
\begin{corollary}
\label{cor:qinf-positive}
Let $0<p<\infty$, $0<\alpha<\infty$ and suppose
$\lambda\in\ell^{\alpha p}$ with
$\|\lambda\|_\infty\leq1$.  Then
\[
 \|S_{\dd}:\B\longrightarrow\Sp p\|
 \leq C_{\alpha,p,f}\|\lambda\|_{\alpha p}^{\alpha}.
\]
\end{corollary}

\begin{corollary}
\label{cor:qinf-positive-range}
Let $0<\alpha,r<\infty$, let $0<p\leq\infty$, and suppose
\(
 \frac1p\leq\frac{\alpha}{r}. 
\)
For every $f\in C^1([-1,1])$ satisfying
$f(0)=0$ and $L_\alpha(f)<\infty$, and every real sequence
$\lambda\in\ell^r$ satisfying $\|\lambda\|_\infty\leq1$, the map
\[
 S_{\Psi_{f,\lambda}}:\Sp\infty=\K\longrightarrow\Sp p
\]
is bounded.
\end{corollary}

\begin{proof}
If $p<\infty$, then $r\leq\alpha p$, and hence
$\ell^r\subseteq\ell^{\alpha p}$.  The assertion therefore follows
from Corollary~\ref{cor:qinf-positive}.

Let $p=\infty$.  Since $0<r/\alpha<\infty$ and
$\lambda\in\ell^r=\ell^{\alpha(r/\alpha)}$,  it follows from 
Corollary~\ref{cor:qinf-positive} (applied with target exponent
$r/\alpha$) that 
\[
 S_{\Psi_{f,\lambda}}:\Sp\infty=\K (\subset \B)
 \longrightarrow\Sp{r/\alpha} \subseteq \Sp\infty.
\]
The proof is complete. 
\end{proof}

\section{The boundedness of $S_{\Psi_{f,\lambda }}: \Sp 1  \to \Sp p$, $0<p<1$}
 \label{s 3}

When $q=1$, the exponent relation $\frac{1}{p}= \frac{\alpha}{r} +\frac{1}{q}$ becomes
\begin{equation}\label{eq:q1}
 \frac1p=1+\frac{\alpha}{r},
 \qquad 0<p<1.
\end{equation}
The present section treats the  boundedness  of $S_{\Psi_{f,\lambda}}$, $\lambda \in \ell^r$,  when the above equality holds.
The results from  \cite{CPSZ} play a crucial role in the proof of Theorem \ref{thm:q1-strong}.  
We also use ideas from  \cite{HSZ,Ricard}.

\begin{theorem}\label{thm:q1-strong}
Let $0<\alpha,r<\infty$, let $p=r/(r+\alpha)$, and suppose
$f\in C^1([-1,1])$ satisfies
$f(0)=0$ and $L_\alpha(f)<\infty$.  If
$\lambda\in\ell^r$ such that 
$\|\lambda\|_\infty\leq1$, then
\begin{equation}\label{eq:q1-strong}
 \|S_{\dd}(A)\|_p
 \leq
 C_{\alpha,p}L_\alpha(f)\,
 \|\lambda\|_r^\alpha\|A\|_1,
 \qquad A\in\Sp1.
\end{equation}
\end{theorem}

For the remainder
of this section, fix $\alpha,r,p$, a real-valued $f$, and $\lambda$ as in
Theorem~\ref{thm:q1-strong}.  Put
\begin{align}\label{sigma}
 \sigma=\frac r\alpha=\frac{p}{1-p},
 \qquad
 \frac1p=1+\frac1\sigma,
\end{align}
and define
\begin{align}\label{def h}
 h(t)=
 \begin{cases}
  f(t)/t,&t\ne0,\\
  0,&t=0.
 \end{cases}
\end{align}
Let $g=(g_j)_{j\geq1}$ be given by
$g_j=h(\lambda_j)$, and let
$G=(G(j,k))_{j,k\geq1}$ be the matrix symbol defined by
\[
 G(j,k):=g_j+g_k.
\]
Consider the residual symbol
\begin{align}\label{def Rf}
 R_f(x,y)
 =
 \begin{cases}
 \displaystyle
 \frac{f(x)-f(y)}{x-y}-h(x)-h(y),&x\ne y,\\[1.2ex]
 f'(x)-2h(x),&x=y.
 \end{cases}
\end{align}
Define the corresponding  residual matrix symbol
$R=(R(j,k))_{j,k\geq1}$ by
\[
 R(j,k):=R_f(\lambda_j,\lambda_k).
\]
As before, for any matrix $A=(a_{jk})$, we write
\(
 S_R(A):=\bigl(R(j,k)a_{jk}\bigr)_{j,k\geq1}.
\)

\begin{lemma}
\label{lem:q1-leading-residual}
With the notation above, we have 
\begin{equation}\label{eq:h-bound}
 |h(t)|\leq\frac{L_\alpha(f)}{1+\alpha}|t|^\alpha,
 \qquad t\in[-1,1],
\end{equation}
and $g\in\ell^\sigma$ satisfies
\begin{equation}\label{eq:g-bound}
 \|g\|_\sigma
 \leq C_\alpha  L_\alpha(f)\|\lambda\|_r^\alpha,
\end{equation}
where $C_\alpha =\frac{1}{1+\alpha}$. 
Moreover, 
\[
 \Psi_{f,\lambda}=G+R,
 \qquad
 S_{\Psi_{f,\lambda}}=S_G+S_R,
\]
and $R(j,k)=0$ whenever $\lambda_j=0$ or $\lambda_k=0$.
Finally, for every $A\in\Sp1$,
\begin{equation}\label{eq:leading-bound}
 \|S_G(A)\|_p
 \leq 2^{1/p}\|g\|_\sigma\|A\|_1.
\end{equation}
\end{lemma}

\begin{proof}
Since $f(0)=0$, it follows that 
\[
 |f(t)|
 \leq
 L_\alpha(f)\int_0^{|t|}s^\alpha\,ds
 =
 \frac{L_\alpha(f)}{1+\alpha}|t|^{1+\alpha}.
\]
By the definition of $h$, this proves \eqref{eq:h-bound}.  Since $\alpha\sigma=r$, it follows that 
\[
 \|g\|_\sigma^\sigma
 \leq
 C_\alpha^\sigma L_\alpha(f)^\sigma
 \sum_{j\geq1}|\lambda_j|^{\alpha\sigma}
 =
 C_\alpha^\sigma L_\alpha(f)^\sigma\|\lambda\|_r^r,
\]
where $C_\alpha = \frac{1}{1+\alpha}$.
This proves \eqref{eq:g-bound}.

If $\lambda_j\ne\lambda_k$, then
\[
 \begin{aligned}
 G(j,k)+R(j,k)
 &=
 h(\lambda_j)+h(\lambda_k)
 +\frac{f(\lambda_j)-f(\lambda_k)}
        {\lambda_j-\lambda_k}
 -h(\lambda_j)-h(\lambda_k)\\
 &=
 \frac{f(\lambda_j)-f(\lambda_k)}
      {\lambda_j-\lambda_k}
 =\Psi_{f,\lambda}(j,k).
 \end{aligned}
\]
If $\lambda_j=\lambda_k$, then
\[
 \begin{aligned}
 G(j,k)+R(j,k)
 &=
 2h(\lambda_j)+f'(\lambda_j)-2h(\lambda_j)\\
 &=f'(\lambda_j)=\Psi_{f,\lambda}(j,k).
 \end{aligned}
\]
Thus, for every $j,k\geq1$,
\[
 \Psi_{f,\lambda}(j,k)=G(j,k)+R(j,k),
 \qquad
 S_{\Psi_{f,\lambda}}=S_G+S_R.
\]

The continuity of $f'$ and $L_\alpha(f)<\infty$ imply
$f'(0)=0$, while $f(0)=h(0)=0$.  Hence,
\[
 R_f(0,0)=f'(0)-2h(0)=0.
\]
If $y\ne0$, then
\[
 R_f(0,y)
 =\frac{f(0)-f(y)}{-y}-h(0)-h(y)
 =\frac{f(y)}y-h(y)=0,
\]
and similarly $R_f(x,0)=0$ for $x\ne0$.  
Therefore,
$R(j,k)=0$ whenever $\lambda_j=0$ or $\lambda_k=0$.

If $A=(a_{jk})$, then the Schur multiplier with symbol $G$ satisfies
\[
 S_G(A)
 =\bigl((g_j+g_k)a_{jk}\bigr)_{j,k\geq1}
 =D_gA+AD_g,
 \qquad D_ge_j=g_je_j.
\]
Since $D_g\in\Sp\sigma$ and $\|D_g\|_\sigma=\|g\|_\sigma$, the
Schatten H\"older inequality
\cite[Theorem~4.2(i)]{FackKosaki}, with
$1/p=1/\sigma+1$, yields
\[
 \|D_gA\|_p\leq\|g\|_\sigma\|A\|_1,
 \qquad
 \|AD_g\|_p\leq\|g\|_\sigma\|A\|_1.
\]
By 
the $p$-triangle inequality  of $\left\|\cdot\right\|_p$
(see \cite{Mc} or see \cite[Proposition~4.6(i)]{FackKosaki} for a more general result), we obtain 
\[
 \|S_G(A)\|_p
 \leq 2^{1/p}\|g\|_\sigma\|A\|_1.
\]
\end{proof}

Retain the notation of Lemma~\ref{lem:q1-leading-residual}.  For
$m\geq0$, let
\begin{align}\label{def Nm}
 I_m=\{j:2^{-m-1}<|\lambda_j|\leq2^{-m}\},
 \qquad P_m=\sum_{j\in I_m}e_{jj},
 \qquad N_m=|I_m|.
\end{align}
Since $\lambda\in\ell^r$, it follows that every $N_m$ is finite.  

We say that a matrix $A=(a_{jk})$ has \emph{finite coordinate
support} if $A=P_FAP_F$ for some finite set
$F\subset\mathbb N$, where $P_F:=\sum_{j\in F}e_{jj}$.
When $A$ has finite coordinate support, the sum
\begin{align}\label{def Q}
 Q:=\sum_{m\geq0}P_m,
\end{align}
is finite
and, for 
$\ell\in\mathbb Z$, we write (note that the following sum is finite because $A$ has finite coordinate support)
\[
 A^{(\ell)}
 :=
 \sum_{m\geq\max\{0,-\ell\}}P_mAP_{m+\ell}.
\]
In other words, $A^{(\ell)}$ is $\ell$-th diagonal of the matrix $A$, parallel to the main diagonal. 
\begingroup
 
Figure~\ref{fig:dyadic-blocks} illustrates both a single dyadic block
$B=P_mAP_n$ and the fixed-lag operator $A^{(\ell)}$ defined above.
\begin{figure}[htbp]
\makeatletter
\renewcommand{\@captionfont}{\normalfont\tiny}
\makeatother
\setlength{\abovecaptionskip}{-1pt}
\centering
\scalebox{0.7}{
\begin{minipage}[t]{0.46\textwidth}
\centering
\begin{tikzpicture}[x=0.52cm,y=-0.52cm,font=\scriptsize]
 \fill[gray!18] (0,3) rectangle (6,4);
 \fill[gray!18] (4,0) rectangle (5,6);
 \fill[gray!62] (4,3) rectangle (5,4);
 \foreach \i in {0,...,6}{
   \draw[black!55,thin] (0,\i)--(6,\i);
   \draw[black!55,thin] (\i,0)--(\i,6);
 }
 \draw[semithick] (4,3) rectangle (5,4);
 \node at (4.5,3.5) {$B$};
 \node at (4.5,-0.45) {$P_n\mathcal H$};
 \node[rotate=90] at (-0.55,3.5) {$P_m\mathcal H$};
\end{tikzpicture}
\par
\textup{(a) The block $B=P_mAP_n$.}
\end{minipage}
\hfill
\begin{minipage}[t]{0.46\textwidth}
\centering
\begin{tikzpicture}[x=0.52cm,y=-0.52cm,font=\scriptsize]
 \foreach \row in {0,...,5}{
   \foreach \col in {0,...,5}{
     \pgfmathtruncatemacro{\gap}{abs(\row-\col)}
     \ifnum\gap<2\relax
       \fill[gray!18] (\col,\row) rectangle ++(1,1);
     \fi
   }
 }
 \foreach \row/\col in {0/2,1/3,2/4,3/5}{
   \fill[gray!62] (\col,\row) rectangle ++(1,1);
   \draw[semithick] (\col,\row) rectangle ++(1,1);
 }
 \foreach \i in {0,...,6}{
   \draw[black!55,thin] (0,\i)--(6,\i);
   \draw[black!55,thin] (\i,0)--(\i,6);
 }
 \draw[very thick] (4,3) rectangle ++(1,1);
 \node at (4.5,3.5) {$B$};
 \node at (3,-0.45) {column shell $n$};
 \node[rotate=90] at (-0.55,3) {row shell $m$};
 \fill[gray!18] (0,6.45) rectangle ++(0.38,0.38);
 \draw (0,6.45) rectangle ++(0.38,0.38);
 \node[right] at (0.5,6.64) {$|m-n|\leq1$};
 \fill[gray!62] (0,7.15) rectangle ++(0.38,0.38);
 \draw (0,7.15) rectangle ++(0.38,0.38);
 \node[right] at (0.5,7.34)
  {$A^{(\ell)}:\ n-m=\ell\quad(\ell=2\text{ shown})$};
\end{tikzpicture}
\par
\textup{(b) Adjacent shells and $A^{(\ell)}$
($\ell=2$ shown).}
\end{minipage}}
\caption{The dyadic block decomposition associated with
$(P_m)_{m\geq0}$.  The light row and column in panel \textup{(a)}
show the actions of $P_m$ and $P_n$; their intersection is
$P_mAP_n$.  In panel \textup{(b)}, the light band consists of all
blocks with $|m-n|\leq1$, and the outlined cell $B$ is one such
block.  The dark blocks form $A^{(\ell)}$ for the fixed lag
$\ell=2$; translating the dark diagonal gives $A^{(\ell)}$ for
every $\ell\in\mathbb Z$.}
\label{fig:dyadic-blocks}
\end{figure}
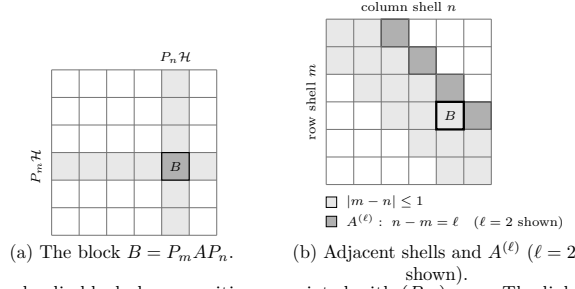
\endgroup

Below, we collect some standard facts. 
\begin{lemma}
\label{lem:q1-lag-decomposition}
Let $A\in\Sp1$ have finite coordinate support and   $Q$ as in \eqref{def Q}. We have
\[
 S_R(A)=S_R(QAQ),
 \qquad
 \|QAQ\|_1\leq\|A\|_1.
\]
If $A=QAQ$, then only finitely many $A^{(\ell)}$ are nonzero, and
\[
 A=\sum_{\ell\in\mathbb Z}A^{(\ell)},
 \qquad
 S_R(A)=\sum_{\ell\in\mathbb Z}S_R(A^{(\ell)}),
 \qquad
 \|A^{(\ell)}\|_1\leq\|A\|_1.
\]
Moreover, if
$a_{m,\ell}:=\|P_mAP_{m+\ell}\|_1$, then
\[
 \|A^{(\ell)}\|_1
 =\sum_{m\geq\max\{0,-\ell\}}a_{m,\ell},
 \qquad
 \|S_R(A^{(\ell)})\|_p^p
 =\sum_{m\geq\max\{0,-\ell\}}
   \|S_R(P_mAP_{m+\ell})\|_p^p.
\]
\end{lemma}

\begin{proof}
The projection $Q$ is the orthogonal projection onto
$\overline{\operatorname{span}}\{e_j:\lambda_j\ne0\}$, and
\[
 Qe_j=\mathbf 1_{\{\lambda_j\ne0\}}e_j.
\]
For every $j,k\geq1$, the $(j,k)$ entry of $S_R(QAQ)$ is
\[
 R(j,k)\mathbf 1_{\{\lambda_j\ne0\}}
 \mathbf 1_{\{\lambda_k\ne0\}}a_{jk}
 =R(j,k)a_{jk},
\]
because $R(j,k)=0$ if $\lambda_j=0$ or $\lambda_k=0$ (see Lemma \ref{lem:q1-leading-residual}).  
Hence, 
\[
 S_R(A)=S_R(QAQ),
 \qquad
 \|QAQ\|_1\leq\|A\|_1.
\]
We may therefore assume $A=QAQ$.


The standard contractivity of
finite block projections (see \cite[Section 3]{CKS92}, 
\cite[(2.6)]{HuangSukochevYu} and \cite[Lemma~6.1]{DPS2016}) therefore
gives
\[
 \|A^{(\ell)}\|_1
 =
 \left\|
  \sum_{m\geq\max\{0,-\ell\}}P_mAP_{m+\ell}
 \right\|_1
 \leq\|A\|_1.
\]

Only finitely many $A^{(\ell)}$ are nonzero  
(because $A$ has finite
coordinate support), and
\[
 A=\sum_{\ell\in\mathbb Z}A^{(\ell)},
 \qquad
 S_R(A)=\sum_{\ell\in\mathbb Z}S_R(A^{(\ell)}).
\]
 For $m\geq\max\{0,-\ell\}$, set
\[
 a_{m,\ell}=\|P_mAP_{m+\ell}\|_1.
\]
Since $P_m$'s are mutually disjoint, it follows that  (note that $|A^{(\ell)}|^p = \displaystyle\sum_{m\ge \max\{0,-\ell\}}|P_m A P_{m+\ell}|^p$)
\begin{equation}\label{eq:lag-direct-sums}
 \|A^{(\ell)}\|_1
 =\sum_{m\geq\max\{0,-\ell\}}a_{m,\ell},
~\mbox{ }
 \|S_R(A^{(\ell)})\|_p^p
 =\sum_{m\geq\max\{0,-\ell\}}
    \|S_R(P_mAP_{m+\ell})\|_p^p.
\end{equation}
This completes the proof. 
\end{proof}

\begingroup
 According to whether the dyadic shells are adjacent
or separated, we will prove two  
 local estimates needed for Theorem~\ref{thm:q1-strong}.  For $m,n\geq0$, put
\begin{align}\label{def: d K beta}
 d_{m,n}=\min\{m,n\}\mbox{ and }
 K_{m,n}=\min\{N_m,N_n\}
.
\end{align}
The following estimate
will be particularly useful in the special case of adjacent-shell estimate, that is the case when    $|m-n|\leq1$ (see the proof of Lemma \ref{lem:q1-fixed-lag}).
\begin{lemma} 
\label{lem:q1-adjacent-block}
If $m,n\geq0$, then, for $p$ satisfying \eqref{eq:q1}, we have
\begin{equation}\label{eq:adjacent-block-estimate}
 \|S_R(B)\|_p
 \leq
 C_pL_\alpha(f)\,
 2^{-\alpha d_{m,n}}K_{m,n}^{1/\sigma}\|B\|_1,~\forall B\in P_m\Sp1P_n.
\end{equation}
\end{lemma}

\begin{proof}
\noindent\textbf{Step 1: }
Here, we prove that 
\begin{equation}\label{eq:adjacent-cpsz}
 \|S_{\Phi_{m,n}}(B)\|_{1,\infty}
 \leq
 C L_\alpha(f)2^{-\alpha d_{m,n}}\|B\|_1,~\forall B\in P_m\Sp1P_n,
\end{equation}
where $C>0$ is the absolute constant (independent of
$\alpha$, $m$ and $n$) in
\cite[Theorem~1.2]{CPSZ}.

The light band in Figure~\ref{fig:dyadic-blocks}\textup{(b)} consists
of the blocks covered by this lemma, and the outlined cell represents
a typical $B\in P_m\Sp1P_n$.  For $j\in I_m$ and $k\in I_n$,
\[
 |\lambda_j|\leq2^{-m}\leq2^{-d_{m,n}},
 \qquad
 |\lambda_k|\leq2^{-n}\leq2^{-d_{m,n}}.
\]
Thus, the spectral values of $H_\lambda$ associated with the nonzero
rows and columns of $B$ lie in
\[
 J_{m,n}:=[-2^{-d_{m,n}},2^{-d_{m,n}}],
\]
and
\[
 \operatorname{Lip}\bigl(f|_{J_{m,n}}\bigr)
 \leq L_\alpha(f)2^{-\alpha d_{m,n}}.
\]
Choose an extension
$\widetilde f_{m,n}:\mathbb R\to\mathbb R$ of
$f|_{J_{m,n}}$ with the same Lipschitz constant (see Kirszbraun's
theorem \cite[Section~2.10]{Federer}).  Following
\cite[Section~2.4]{CPSZ}, define
\[
 \varphi_{m,n}(x,y)
 =
 \begin{cases}
 \displaystyle
 \frac{\widetilde f_{m,n}(x)-\widetilde f_{m,n}(y)}{x-y},
 &x\ne y,\\[1.2ex]
 0,&x=y,
 \end{cases}
\]
and
\begin{align}\label{def Phi}
 \Phi_{m,n}(j,k)
 :=\varphi_{m,n}(\lambda_j,\lambda_k),
 \qquad j,k\geq1.
\end{align}
The claim \eqref{eq:adjacent-cpsz} follows from  \cite[Theorem~1.2]{CPSZ}. 
Indeed, 
let $E_{H_\lambda}$ be the spectral measure of
$$H_\lambda=\sum_{j\geq1}\lambda_je_{jj}. $$  The associated double
operator integral is
\[
 T_{\varphi_{m,n}}^{H_\lambda,H_\lambda}(X)
 :=\int_{\mathbb R}\!\int_{\mathbb R}
 \varphi_{m,n}(x,y)\,dE_{H_\lambda}(x)X\,dE_{H_\lambda}(y).
\]
Since $E_{H_\lambda}(\{\lambda_j\})e_{jk}
E_{H_\lambda}(\{\lambda_k\})=e_{jk}$, one has
\[
 T_{\varphi_{m,n}}^{H_\lambda,H_\lambda}(e_{jk})
 =\varphi_{m,n}(\lambda_j,\lambda_k)e_{jk}
 =\Phi_{m,n}(j,k)e_{jk}.
\]
Consequently, writing
$B=\sum_{j\in I_m,\,k\in I_n}b_{jk}e_{jk}$,
\begin{equation}\label{eq:doi-schur-identification}
 T_{\varphi_{m,n}}^{H_\lambda,H_\lambda}(B)
 =\sum_{j\in I_m,\,k\in I_n}
   \Phi_{m,n}(j,k)b_{jk}e_{jk}
 =S_{\Phi_{m,n}}(B).
\end{equation}
We have
\begin{equation*} 
 \|S_{\Phi_{m,n}}(B)\|_{1,\infty}
 \stackrel{\eqref{eq:doi-schur-identification}}{=}\|T_{\varphi_{m,n}}^{H_\lambda,H_\lambda}(B)\|_{1,\infty}
 \stackrel{\mbox{\tiny \cite[Theorem~1.2]{CPSZ}}}{\leq}
 C L_\alpha(f)2^{-\alpha d_{m,n}}\|B\|_1.
\end{equation*}
This proves \eqref{eq:adjacent-cpsz}.

\noindent\textbf{Step 2:}
Here, we prove that 
\begin{equation}\label{eq:adjacent-dd-weak}
 \|S_{\Psi_{f,\lambda}}(B)\|_{1,\infty}
 \leq
 2(C+1)L_\alpha(f)2^{-\alpha d_{m,n}}\|B\|_1,~\forall B\in P_m\Sp1P_n.
\end{equation}

For $j\in I_m$ and $k\in I_n$ with
$\lambda_j\ne\lambda_k$, the equality
$\widetilde f_{m,n}=f$ on $J_{m,n}$ implies
\[
 \Phi_{m,n}(j,k)=\Psi_{f,\lambda}(j,k).
\]
If $\lambda_j=\lambda_k=\xi$, then
$\Phi_{m,n}(j,k)=0$, whereas
$\Psi_{f,\lambda}(j,k)=f'(\xi)$.  For every eigenvalue $\xi$ (i.e., $\lambda_j$) of
$H_\lambda$, let
\[
 E_\xi:=\sum_{\{j:\lambda_j=\xi\}}e_{jj},
\]
and set
\[
 \Xi_{m,n}
 :=
 \{\xi:\lambda_j=\lambda_k=\xi
       \text{ for some }j\in I_m,\ k\in I_n\}.
\]
The set $\Xi_{m,n}$ is finite because $I_m$ and $I_n$ are finite (see \eqref{def Nm}).
Define
\begin{align}\label{def D}
 \mathcal D_{m,n}(B)
 :=
 \sum_{\xi\in\Xi_{m,n}} f'(\xi)E_\xi B E_\xi.
\end{align}
For $j\in I_m$ and $k\in I_n$,
\[
 (E_\xi B E_\xi)_{jk}
 =
 \mathbf 1_{\{\lambda_j=\xi\}}
 \mathbf 1_{\{\lambda_k=\xi\}}b_{jk},
\]
so
\begin{align*}
 \bigl[\mathcal D_{m,n}(B)\bigr]_{jk}
 =
 \sum_{\xi\in\Xi_{m,n}}
 f'(\xi)
 \mathbf 1_{\{\lambda_j=\xi\}}
 \mathbf 1_{\{\lambda_k=\xi\}}b_{jk}=
 \mathbf 1_{\{\lambda_j=\lambda_k\}}
 f'(\lambda_j)b_{jk}.
\end{align*}
It follows that
\begin{align*}
 \bigl[S_{\Phi_{m,n}}(B)+\mathcal D_{m,n}(B)\bigr]_{jk}
 &~=~
 \left(
  \Phi_{m,n}(j,k)
  +\mathbf 1_{\{\lambda_j=\lambda_k\}}f'(\lambda_j)
 \right)b_{jk}\\
 &\stackrel{\eqref{def Phi}}{=}
 \begin{cases}
  \displaystyle
  \frac{f(\lambda_j)-f(\lambda_k)}
       {\lambda_j-\lambda_k}\,b_{jk},
       &\lambda_j\ne\lambda_k,\\[1.2ex]
  f'(\lambda_j)b_{jk},&\lambda_j=\lambda_k,
 \end{cases}\\
 &~=~\Psi_{f,\lambda}(j,k)b_{jk}.
\end{align*}
Since $B=P_mBP_n$ (in other words, 
all entries outside $I_m\times I_n$ vanish), it follows that 
\begin{equation}\label{eq:adjacent-diagonal-decomposition}
 S_{\Psi_{f,\lambda}}(B)
 =S_{\Phi_{m,n}}(B)+\mathcal D_{m,n}(B).
\end{equation}
If $E_\xi B E_\xi\ne0$, then
$|\xi|\leq2^{-d_{m,n}}$ and therefore
\[
 |f'(\xi)|\leq L_\alpha(f)2^{-\alpha d_{m,n}}.
\]
Since the mapping 
$X\mapsto\sum_{\xi\in\Xi_{m,n}}E_\xi X E_\xi$ is contractive on
$\Sp1$,  see e.g. \cite[Section 3]{CKS92}, 
\cite[(2.6)]{HuangSukochevYu} and \cite[Lemma~6.1]{DPS2016}.
Thus, 
\begin{align}\label{Dmn bound}
\|\mathcal D_{m,n}(B)\|_{1,\infty  }\stackrel{\tiny \mbox{\cite[p. 217]{BS}}}{\le} 
 \|\mathcal D_{m,n}(B)\|_1
&\stackrel{\eqref{def D}}{\leq}
 L_\alpha(f)2^{-\alpha d_{m,n}}
 \left\|\sum_{\xi\in\Xi_{m,n}}E_\xi B E_\xi\right\|_1 \nonumber \\
 &~\leq~
 L_\alpha(f)2^{-\alpha d_{m,n}}\|B\|_1.
\end{align}
By
\eqref{eq:adjacent-diagonal-decomposition}   and the quasi-triangle inequality in
$\Sp{1,\infty}$ \cite[Section~2.2]{CPSZ}, we obtain
\begin{align*}
 \|S_{\Psi_{f,\lambda}}(B)\|_{1,\infty}
 &\leq
 2\bigl(
   \|S_{\Phi_{m,n}}(B)\|_{1,\infty}
   +\|\mathcal D_{m,n}(B)\|_{1,\infty}
 \bigr)\\
 &\stackrel{\eqref{eq:adjacent-cpsz},\eqref{Dmn bound}}{\leq}
 2(C+1)L_\alpha(f)2^{-\alpha d_{m,n}}\|B\|_1.
\end{align*}
This proves \eqref{eq:adjacent-dd-weak}.

\noindent\textbf{Step 3:} Here, we prove that 
\begin{equation}\label{eq:near-weak}
 \|S_R(B)\|_{1,\infty}
 \leq
 4(C+2)L_\alpha(f)2^{-\alpha d_{m,n}}\|B\|_1,~\forall B\in P_m\Sp1P_n.
\end{equation}

The definition $g_j=h(\lambda_j)$ and \eqref{eq:h-bound} give 
\[
 \begin{aligned}
 \|D_gP_m\|_\infty
 &=\sup_{j\in I_m}|h(\lambda_j)|
 \leq\frac{L_\alpha(f)}{1+\alpha}
       \sup_{j\in I_m}|\lambda_j|^\alpha
 \leq\frac{L_\alpha(f)}{1+\alpha}2^{-\alpha m},\\
 \|D_gP_n\|_\infty
 &=\sup_{k\in I_n}|h(\lambda_k)|
 \leq\frac{L_\alpha(f)}{1+\alpha}2^{-\alpha n}.
 \end{aligned}
\]
Since $B=P_mBP_n$, it follows that 
$D_gB=(D_gP_m)B$ and $BD_g=B(P_nD_g)$.  
Hence,
\[
 \begin{aligned}
 \|S_G(B)\|_1
 &=\|D_gB+BD_g\|_1\\
 &\leq
 \bigl(\|D_gP_m\|_\infty+\|D_gP_n\|_\infty\bigr)\|B\|_1\\
 &\leq
 \frac{2}{1+\alpha}
 L_\alpha(f)2^{-\alpha d_{m,n}}\|B\|_1.
 \end{aligned}
\]
By
$S_R(B)\stackrel{\tiny \mbox{Lemma \ref{lem:q1-leading-residual}}}{=}S_{\Psi_{f,\lambda}}(B)-S_G(B)$, the latter  inequality and 
\eqref{eq:adjacent-dd-weak}, we have 
\begin{align*}
 \|S_R(B)\|_{1,\infty}
 \leq
 2\bigl(
   \|S_{\Psi_{f,\lambda}}(B)\|_{1,\infty}
   +\|S_G(B)\|_{1,\infty}
 \bigr)\leq
 4(C+1+\frac{1}{1+\alpha})L_\alpha(f)2^{-\alpha d_{m,n}}\|B\|_1.
\end{align*}

\noindent\textbf{Step 4:} 
By the definition of
Schur   multiplication, we have 
\[
 S_R(B)=P_mS_R(B)P_n,
 \qquad
 \operatorname{rank}S_R(B)\leq K_{m,n}.
\]
If a matrix $X$ is such that $\operatorname{rank}X\leq K$ for some integer $K$, then
$s_j(X)\leq j^{-1}\|X\|_{1,\infty}$ for $1\leq j\leq K$.
In particular, 
\[
 \begin{aligned}
 \|X\|_p^p
 &\leq\|X\|_{1,\infty}^p\sum_{j=1}^Kj^{-p}
 \leq C_p^pK^{1-p}\|X\|_{1,\infty}^p,\\
 \|X\|_p
 &\leq C_pK^{1/p-1}\|X\|_{1,\infty}
 =C_pK^{1/\sigma}\|X\|_{1,\infty}.
 \end{aligned}
\]
Applying this inequality to $X=S_R(B)$ and using
\eqref{eq:near-weak}, 
we obtain
\eqref{eq:adjacent-block-estimate}.
\end{proof}

\begin{lemma}[Separated-shell estimate]
\label{lem:q1-separated-block}
If $m,n\geq0$ and $|m-n|\geq2$, then,  

\begin{equation}\label{eq:separated-block-estimate}
 \|S_R(B)\|_p
 \leq
 C_{\alpha,p}L_\alpha(f)\,
 2^{-\alpha d_{m,n}}2^{-\beta|m-n|}
 K_{m,n}^{1/\sigma}\|B\|_1,~\forall B\in P_m\Sp1P_n, 
\end{equation}
where $ \beta=\min\{1,\alpha\}$. 
\end{lemma}

\begin{proof}
First suppose $n=m+\ell$ with $\ell\geq2$.  
Assume that 
\[
 2^{-m-1}<|x|\leq2^{-m},
 \qquad
 2^{-n-1}<|y|\leq2^{-n}.
\]
In particular, $x,y\ne0$.  Put $z=y/x$; then
$|z|\leq2^{1-\ell}\leq1/2$.  Recalling that
$h(t)=f(t)/t$ for $t\ne0$ (see \eqref{def h} above), we compute
\[
 \begin{aligned}
 R_f(x,y)
 &\stackrel{\eqref{def Rf}}{=}\frac{xh(x)-yh(y)}{x-y}-h(x)-h(y)\\
 &~=~\frac{h(x)-zh(y)}{1-z}-h(x)-h(y)\\
 &~=~\frac{zh(x)-h(y)}{1-z}\\
 &~=~\sum_{s\geq0}\bigl(h(x)z^{s+1}-h(y)z^s\bigr)\\
 &~=~\sum_{s\geq0}
 \left(
 h(x)x^{-s-1}y^{s+1}-x^{-s}h(y)y^s
 \right).
 \end{aligned}
\]
 The series is absolutely convergent because $|z|\leq1/2$. %
 Now, setting 
\[
\frac{1}{2} \ge \rho : =2^{1-\ell} = \frac{2^{-m-\ell}}{2^{-m-1}} \ge 
  \left|\frac{\lambda_k}{\lambda_j}\right|,
  \qquad(j\in I_m,\ k\in I_n),
 \]
 we have\footnote{For $s\geq0$, define the scalar families
 $u^{(s)}=(h(\lambda_j)\lambda_j^{-s-1})_{j\in I_m}$,
 $v^{(s)}=(\lambda_k^{s+1})_{k\in I_n}$,
 $\widetilde u^{(s)}=(\lambda_j^{-s})_{j\in I_m}$, and
 $\widetilde v^{(s)}=(h(\lambda_k)\lambda_k^s)_{k\in I_n}$.
 For a scalar family $a=(a_i)_{i\in I}$, write
 $D_a=\sum_{i\in I}a_ie_{ii}$ for its associated diagonal operator.
 Thus the term with index $s$ acts on $B$ as
 $D_{u^{(s)}}BD_{v^{(s)}}-
 D_{\widetilde u^{(s)}}BD_{\widetilde v^{(s)}}$.}
 \begin{eqnarray}
  \|S_R:P_m\Sp1P_n\to\Sp1\|
  &\leq&\sum_{s\geq0}
  \Bigl(
    \|D_{u^{(s)}}\|_\infty\|D_{v^{(s)}}\|_\infty
   +\|D_{\widetilde u^{(s)}}\|_\infty
    \|D_{\widetilde v^{(s)}}\|_\infty
  \Bigr)\notag\\
  &=&\sum_{s\geq0}\Biggl(
    \sup_{\substack{j\in I_m\\ k\in I_n}}
    |h(\lambda_j)|
    \left|\frac{\lambda_k}{\lambda_j}\right|^{s+1}
   +\sup_{\substack{j\in I_m\\ k\in I_n}}
    |h(\lambda_k)|
    \left|\frac{\lambda_k}{\lambda_j}\right|^s
  \Biggr)\notag\\
  &\stackrel{\eqref{eq:h-bound}}{\leq}&
  C_\alpha L_\alpha(f)
  \sum_{s\geq0}
  \bigl(2^{-\alpha m}\rho^{s+1}
       +2^{-\alpha(m+\ell)}\rho^s\bigr)\notag\\
  &=&C_\alpha L_\alpha(f)2^{-\alpha m}
  \sum_{s\geq0}
  \bigl(\rho^{s+1}+2^{-\alpha\ell}\rho^s\bigr)\notag\\
  &=&C_\alpha L_\alpha(f)2^{-\alpha m}
  \frac{\rho+2^{-\alpha\ell}}{1-\rho}\notag\\
  &=&C_\alpha L_\alpha(f)2^{-\alpha m}
  \frac{2^{1-\ell}+2^{-\alpha\ell}}{1-2^{1-\ell}}\notag\\
  &\stackrel{\ell\geq2}{\leq}&
  2C_\alpha L_\alpha(f)2^{-\alpha m}
  \bigl(2^{1-\ell}+2^{-\alpha\ell}\bigr)\notag\\
  &\stackrel{\beta=\min\{1,\alpha\}}{\leq}&
  C_\alpha L_\alpha(f)2^{-\alpha m}2^{-\beta\ell}.
  \label{eq:far-trace}
 \end{eqnarray}

Since
$S_R(B)=P_mS_R(B)P_n$, it follows that $S_R(B)$ has rank at most $K_{m,n}$.
For every operator $X$ of rank at most $K$, 
we have\footnote{Indeed, Jensen's inequality\cite{Jensen}, applied to the concave function
$t\mapsto t^p$, gives
\(
 \frac1K\sum_{j=1}^Ks_j(X)^p
 \leq
 \left(\frac1K\sum_{j=1}^Ks_j(X)\right)^p.
\)
}
\[
 \|X\|_p\leq K^{1/p-1}\|X\|_1
 =K^{1/\sigma}\|X\|_1.
\]
Inequality \eqref{eq:separated-block-estimate} when $n-m\geq2$ is proved by 
applying the latter inequality to $X=S_R(B)$ and using \eqref{eq:far-trace}.

The case when
 $m-n\geq2$ follows from the above case by taking   $B^{\mathsf T}$,  the transpose relative to
 the fixed basis. 
\end{proof}

Recall that $ \beta=\min\{1,\alpha\}$. 
Define
\begin{align}\label{def eta}
 \eta_\ell=
 \begin{cases}
  1,&|\ell|\leq1,\\
  2^{-\beta|\ell|},&|\ell|\geq2.
 \end{cases}
\end{align}
\endgroup

\begin{lemma}
\label{lem:q1-fixed-lag} Let $0< p<1$ satisfy \eqref{eq:q1} and $r,\alpha $ be as in Theorem \ref{thm:q1-strong}.
Let $A\in\Sp1$ have finite coordinate support and satisfy $A=QAQ$.
For every $\ell\in\mathbb Z$, we have
\begin{equation}\label{eq:one-lag}
 \|S_R(A^{(\ell)})\|_p
 \leq
 C_{\alpha,p}L_\alpha(f)\eta_\ell
 \|\lambda\|_r^\alpha\|A\|_1.
\end{equation}
\end{lemma}

\begin{proof}
Fix $\ell\in\mathbb Z$.  Recall that following decompositions
\[
 A^{(\ell)}=
 \sum_{m\geq\max\{0,-\ell\}}P_mAP_{m+\ell},
 \qquad
 S_R(A^{(\ell)})=
 \sum_{m\geq\max\{0,-\ell\}}S_R(P_mAP_{m+\ell}),
\]
and \eqref{eq:lag-direct-sums} gives
\[
 \|S_R(A^{(\ell)})\|_p^p
 =\sum_{m\geq\max\{0,-\ell\}}
   \|S_R(P_mAP_{m+\ell})\|_p^p.
\]
For $m\geq\max\{0,-\ell\}$, set
\[
 w_{m,\ell}
:=
 2^{-\alpha\min\{m,m+\ell\}}
 K_{m,m+\ell}^{1/\sigma}.
\]
Since $\alpha\sigma=r$ (see   \eqref{sigma}), it follows that 
\begin{align}
 \sum_{m\geq\max\{0,-\ell\}} w_{m,\ell}^{\sigma}
 &~=~
 \sum_{m\geq\max\{0,-\ell\}}
 2^{-r\min\{m,m+\ell\}}K_{m,m+\ell}\notag\\
 &\stackrel{\eqref{def: d K beta}}{\leq}
 \sum_{s\geq0}2^{-rs}N_s
\stackrel{\eqref{def Nm}}{\leq} \sum_{j\ge 1} (2\lambda _j) ^{r}  = 2^r\|\lambda\|_r^r.
 \label{eq:weight-sum}
\end{align}
The scalar H\"older inequality associated with
$1/p=1+1/\sigma$ says
\[
\begin{aligned}
 \left(
 \sum_{m\geq\max\{0,-\ell\}}
 \bigl(w_{m,\ell}\|P_mAP_{m+\ell}\|_1\bigr)^p
 \right)^{1/p}
 &\leq
 \left(
 \sum_{m\geq\max\{0,-\ell\}}w_{m,\ell}^{\sigma}
 \right)^{1/\sigma}\\
 &\quad\times
 \sum_{m\geq\max\{0,-\ell\}}\|P_mAP_{m+\ell}\|_1.
\end{aligned}
\]
For $m\geq\max\{0,-\ell\}$, apply
Lemma~\ref{lem:q1-adjacent-block} when $|\ell|\leq1$ and
Lemma~\ref{lem:q1-separated-block} when $|\ell|\geq2$.  In either case,
\[
 \|S_R(P_mAP_{m+\ell})\|_p
 \leq C_{\alpha,p}L_\alpha(f)\eta_\ell
 w_{m,\ell}\|P_mAP_{m+\ell}\|_1.
\]
Consequently,
\begin{align*}
 \|S_R(A^{(\ell)})\|_p
 &=\left(
   \sum_{m\geq\max\{0,-\ell\}}
   \|S_R(P_mAP_{m+\ell})\|_p^p
   \right)^{1/p}\\
 &\leq C_{\alpha,p}L_\alpha(f)\eta_\ell
   \left(
   \sum_{m\geq\max\{0,-\ell\}}
   \bigl(w_{m,\ell}\|P_mAP_{m+\ell}\|_1\bigr)^p
   \right)^{1/p}\\
 &\leq C_{\alpha,p}L_\alpha(f)\eta_\ell
   \left(
   \sum_{m\geq\max\{0,-\ell\}}w_{m,\ell}^{\sigma}
   \right)^{1/\sigma}
   \sum_{m\geq\max\{0,-\ell\}}\|P_mAP_{m+\ell}\|_1\\
 &\leq 2^{r/\sigma} C_{\alpha,p}L_\alpha(f)\eta_\ell
   \|\lambda\|_r^\alpha\|A^{(\ell)}\|_1\\
 &\leq 2^{\alpha }C_{\alpha,p}L_\alpha(f)\eta_\ell
   \|\lambda\|_r^\alpha\|A\|_1.
\end{align*}
%
This proves \eqref{eq:one-lag}.
\end{proof}

Below, we obtain the estimate  of  $S_R(A)$ for general $A\in \Sp 1$, extending the result of Lemma \ref{lem:q1-fixed-lag}.
\begin{lemma}
\label{lem:q1-residual}
For every $A\in\Sp1$,
\begin{equation}\label{eq:residual-goal}
 \|S_R(A)\|_p
 \leq
 C_{\alpha,p}L_\alpha(f)\,
 \|\lambda\|_r^\alpha\|A\|_1.
\end{equation}
\end{lemma}

\begin{proof}
Suppose first that $A$ has finite coordinate support.
Lemma~\ref{lem:q1-lag-decomposition} allows us to replace $A$ by
$QAQ$ without changing $S_R(A)$ or increasing $\|A\|_1$. 
\begingroup 
Since $\beta>0$ and $p>0$, it follows from  the definition of $\eta_\ell$ (see \eqref{def eta}) that
\[
 \sum_{\ell\in\mathbb Z}\eta_\ell^p
 =3+2\sum_{\ell=2}^\infty2^{-\beta p\ell}
 =3+\frac{2^{1-2\beta p}}{1-2^{-\beta p}}
 <\infty.
\]
By the $p$-subadditivity of the Schatten quasi-norm
\cite[Proposition~4.6(i)]{FackKosaki} and 
 Lemma~\ref{lem:q1-fixed-lag}, we obtain
\begin{align*}
 \|S_R(A)\|_p^p
 &\leq\sum_{\ell\in\mathbb Z}
        \|S_R(A^{(\ell)})\|_p^p\\
 &\leq C_{\alpha,p}^pL_\alpha(f)^p
   \|\lambda\|_r^{\alpha p}\|A\|_1^p
   \sum_{\ell\in\mathbb Z}\eta_\ell^p\\
 &\leq \widetilde C_{\alpha,p}^{p}L_\alpha(f)^p
   \|\lambda\|_r^{\alpha p}\|A\|_1^p.
\end{align*}
\endgroup
This proves   \eqref{eq:residual-goal} for finite
coordinate matrices.
By the density of finite coordinate matrices  in $\Sp1$, and
the completeness of $\Sp p$,  the estimate extends to all $A\in\Sp1$.
\end{proof}

\begin{proof}[Proof of Theorem~\ref{thm:q1-strong}]
By Lemma~\ref{lem:q1-leading-residual},
$S_{\Psi_{f,\lambda}}=S_G+S_R$. By
\eqref{eq:g-bound} and \eqref{eq:leading-bound}, we have 
\[
 \|S_G(A)\|_p
 \leq
 C_{\alpha,p}L_\alpha(f)
 \|\lambda\|_r^\alpha\|A\|_1.
\]
Lemma~\ref{lem:q1-residual} gives the same estimate for $S_R(A)$.
The $p$-triangle inequality for $\Sp p$
\cite[Proposition~4.6(i)]{FackKosaki} now yields
\eqref{eq:q1-strong}.
\end{proof}



\begin{remark}\label{rem:q-greater-than-one}
The argument used above also applies when
$1<q<\infty$.  Let $p>0$ satisfy
\[
 \frac1p=\frac1q+\frac{\alpha}{r},
\]
and put
\[
 s=\frac r\alpha,
 \qquad
 \frac1p=\frac1q+\frac1s.
\]
Retain the notation
\[
 \beta=\min\{1,\alpha\},
 \qquad
 \eta_\ell=
 \begin{cases}
  1,&|\ell|\leq1,\\
  2^{-\beta|\ell|},&|\ell|\geq2.
 \end{cases}
\]

In the proof of Lemma~\ref{lem:q1-adjacent-block}, replace the
weak $(1,1)$ estimate by the strong $(q,q)$ estimate for
Lipschitz divided differences
\cite[Theorem~7]{PotapovSukochev2011}, which holds
for every $1<q<\infty$.
The treatment of the
diagonal correction is unchanged.  The separated-shell
expansion in   Lemma~\ref{lem:q1-separated-block} also acts
boundedly on $\Sp q$ and gives the same factor
$\eta_{m-n}$ (see \eqref{def eta}).  Hence, for $B\in P_m\Sp qP_n$,
\[
 \|S_R(B)\|_q
 \leq
 C_{\alpha,q}L_\alpha(f)
 2^{-\alpha d_{m,n}}\eta_{m-n}\|B\|_q.
\]
Since $p<q$ and
$\operatorname{rank}S_R(B)\leq K_{m,n}$, it follows that
\[
 \begin{aligned}
 \|S_R(B)\|_p
 \leq
 K_{m,n}^{1/p-1/q}\|S_R(B)\|_q\leq
 C_{\alpha,q}L_\alpha(f)
 2^{-\alpha d_{m,n}}\eta_{m-n}
 K_{m,n}^{1/s}\|B\|_q.
 \end{aligned}
\]

For a fixed $\ell\in\mathbb Z$, define
\[
 w_{m,\ell}
 :=
 2^{-\alpha\min\{m,m+\ell\}}
 K_{m,m+\ell}^{1/s},
 \qquad m\geq\max\{0,-\ell\}.
\]
The proof of Lemma~\ref{lem:q1-fixed-lag} now applies with the
scalar H\"older inequality
\[
 \left(\sum_m(w_{m,\ell}b_m)^p\right)^{1/p}
 \leq
 \left(\sum_mw_{m,\ell}^s\right)^{1/s}
 \left(\sum_mb_m^q\right)^{1/q}
\]
in place of its $q=1$ version.  The same argument
gives
\[
 \sum_{m\geq\max\{0,-\ell\}}w_{m,\ell}^s
 \leq
 2^r\|\lambda\|_r^r.
\]
Consequently,
\[
 \|S_R(A^{(\ell)})\|_p
 \leq
 C_{\alpha,q}L_\alpha(f)\eta_\ell
 \|\lambda\|_r^\alpha\|A^{(\ell)}\|_q.
\]
The proof
of Lemma~\ref{lem:q1-residual} applies without further change,
using the ordinary triangle inequality when $p\geq1$ and the
$p$-triangle inequality when $0<p<1$.  Thus,
\[
 \|S_R(A)\|_p
 \leq
 C_{\alpha,p,q}L_\alpha(f)
 \|\lambda\|_r^\alpha\|A\|_q.
\]

Finally, the estimate of the  term
$S_G(A)=D_gA+AD_g$ is identical, except that the Schatten
H\"older inequality is now used with
\(
 \frac1p=\frac1q+\frac1s.
\)
Combining the estimates for $S_G$ and $S_R$, we obtain
\[
 \|S_{\Psi_{f,\lambda}}(A)\|_p
 \leq
 C_{\alpha,p,q}L_\alpha(f)
 \|\lambda\|_r^\alpha\|A\|_q,
 \qquad A\in\Sp q.
\]
\end{remark}

\begin{corollary}
\label{cor:finite-positive}
Let $0<\alpha,r<\infty$ and $0<p,q<\infty$.  If
\begin{equation}\label{eq:finite-positive}
 \frac1p\leq
 \frac{\alpha}{r}+\min\!\left\{1,\frac1q\right\},
\end{equation}
then, for every $f\in C^1([-1,1])$ satisfying
$f(0)=0$ and $L_\alpha(f)<\infty$, and every real sequence
$\lambda\in\ell^r$ satisfying $\|\lambda\|_\infty\leq1$, the map
\[
 S_{\Psi_{f,\lambda}}:\Sp q\longrightarrow\Sp p
\]
is bounded, and
\[
 \|S_{\Psi_{f,\lambda}}:\Sp q\to\Sp p\|
 \leq
 C_{\alpha,p,q,r,f}\|\lambda\|_r^\alpha.
\]
\end{corollary}

\begin{proof}
Suppose first that $q>1$, and define $p_0$ by
\[
 \frac1{p_0}=\frac{\alpha}{r}+\frac1q.
\]
By Theorem~\ref{thm:PST}, $S_{\Psi_{f,\lambda}}$ maps $\Sp q$ into $\Sp{p_0}$.
Condition \eqref{eq:finite-positive} gives  
$\Sp{p_0}\subseteq\Sp p$.

If $0<q\leq1$, put $p_-=r/(r+\alpha)$.  Then
\[S_{\Psi_{f,\lambda}}: 
 \Sp q\subseteq\Sp1
 \xrightarrow{\ \rm  Theorem~\ref{thm:q1-strong}\ }
 \Sp{p_-}
 \subseteq\Sp p,
\]
where  the last
inclusion follows from \eqref{eq:finite-positive} (i.e., $p\ge p_-$).
\end{proof}

\section{Sharpness of  the assumption in Theorem \ref{thm:main-classification}} 
\label{sec:sharp-classification}

In this section, we show that the condition \eqref{eq:full-condition} in the second part of  Theorem \ref{thm:main-classification} is sharp.

 For
$0<p<1$, Aleksandrov and Peller
\cite[Theorem~4.1]{AleksandrovPeller} (see also \cite[Lemma 2.2.3]{McS}) proved that the symbol
$(d_j\delta_{jk})_{j,k\geq1}$ defines a bounded Schur multiplier on
$\Sp p$ if and only if $d\in\ell^{p/(1-p)}$; in that case,
\[
 \bigl\|S_{(d_j\delta_{jk})_{j,k\geq1}}:\Sp p\to\Sp p\bigr\|
 =
 \|d\|_{\ell^{p/(1-p)}}.
\]

For $q=\infty$ and $1\leq p<\infty$, the necessity of
$1/p\leq\alpha/r$ is given in
\cite[Proposition~3.1(i)]{Arazy}.  Indeed, suppose that
$1/p>\alpha/r$, so that $\alpha p<r$, and choose a positive sequence
$\lambda\in\ell^r\setminus\ell^{\alpha p}$ with
$\|\lambda\|_\infty<1$.  For
\[
 f_0(t):=t|t|^\alpha
 \qquad(-1\leq t\leq1),
\]
we have $f_0'(t)=(\alpha+1)|t|^\alpha$.  Let $P_N$ be the
orthogonal projection onto the first $N$ basis vectors.  Then
$P_N\in C_\infty=\K$ and $\|P_N\|_\infty=1$, whereas
\[
 \bigl\|S_{\Psi_{f_0,\lambda}}(P_N)\bigr\|_p
 =(\alpha+1)
 \left(\sum_{j=1}^N\lambda_j^{\alpha p}\right)^{1/p}
 \longrightarrow\infty.
\]
Thus $S_{\Psi_{f_0,\lambda}}:C_\infty=\K\to\Sp p$ is unbounded.
This calculation does not use $p\geq1$, and hence proves the same
necessity for $q=\infty$ and every $0<p<\infty$.  We now treat the
remaining input indices $0<q<\infty$.

\begin{lemma} 
\label{lem:diagonal-realization}
Let $0<\alpha,r<\infty$, and let $d=(d_j)_{j\geq1}$ be a strictly
decreasing positive sequence such that
$d\in\ell^{r/\alpha}$ and $\|d\|_\infty\leq1$.  Set
\[
 \lambda_j:=\frac12 d_j^{1/\alpha},\qquad j\geq1.
\]
Then $\lambda\in\ell^r$, with $\|\lambda\|_\infty\leq1/2$, and
there is a real-valued $f\in C^1([-1,1])$, supported in $[0,1]$,
such that
\[
 f(0)=0,\qquad |f'(t)|\lesssim_\alpha |t|^\alpha,
 \qquad
 \Psi_{f,\lambda}(j,k)=d_j\delta_{jk}.
\]
\end{lemma}

\begin{proof}
The definition of $\lambda$ gives
\[
 \sum_{j\geq1}|\lambda_j|^r
 =2^{-r}\sum_{j\geq1}d_j^{r/\alpha}<\infty,
 \qquad
 \|\lambda\|_\infty
 =\frac12\|d\|_\infty^{1/\alpha}\leq\frac12.
\]
Choose $\delta_j$ be such that 
$0<\delta_j<\lambda_j/4$ and 
intervals
$I_j=[\lambda_j-\delta_j,\lambda_j+\delta_j]$ are pairwise disjoint.  
Fix a real-valued function
$\phi\in C_c^\infty((-1,1))$ such that
$\phi(0)=0$ and $\phi'(0)=1$, and define, for $x\in[-1,1]$,
\[
 f(x):=\sum_{j\geq1}d_j\delta_j
 \phi\!\left(\frac{x-\lambda_j}{\delta_j}\right).
\]
The supports of
$ \phi\!\left(\frac{\cdot-\lambda_j}{\delta_j}\right)$ 
lie in $I_j = [\lambda_j -\delta_j,\lambda_j+\delta_j]$ (equivalently, $-1\le \frac{x-\lambda_j}{\delta_j} \le 1$), which
are
disjoint.
Since $\lambda_j\leq1/2$ and $\delta_j<\lambda_j/4$, all the
intervals $I_j$ lie in $(0,5/8)$.  
Thus $f$ is supported in
$[0,5/8]\subset[0,1]$.    If $x\in I_j$, then
$|x|\geq3\lambda_j/4$, and hence
\[
 |f'(x)|
 \leq\|\phi'\|_\infty d_j
 =2^\alpha\|\phi'\|_\infty\lambda_j^\alpha
 \leq
 \|\phi'\|_\infty
 \left(\frac83|x|\right)^\alpha.
\]
Moreover, $\delta_j<\lambda_j/4$ gives
\[
 \frac{|f(x)|}{|x|}
 \leq
 \|\phi\|_\infty
 \frac{d_j\delta_j}{3\lambda_j/4}
 \leq\frac{\|\phi\|_\infty}{3}d_j.
\]
For $x\ne0$ outside the intervals $I_j$, both $f(x)$ and
$f'(x)$ vanish.  Since
$d_j\to0$ as $j\to \infty $, these estimates show that
$f(x)/x\to0$ and $f'(x)\to0$ as $x\to0$. 
Consequently,
$f$ is differentiable at $0$, with $f'(0)=0$, and $f'$ is continuous
at $0$.  Hence, $f\in C^1([-1,1])$.
Finally, since 
$f(\lambda_j)=0$ and $f'(\lambda_j)=d_j$, it follows from the definition of $\Psi_{f,\lambda }$ (see \eqref{def Psi}) that 
\[
 \Psi_{f,\lambda}(j,k)=d_j\delta_{jk}.
\]
\end{proof}

\begin{theorem} \label{thm:sharp-failure}
Let $0<\alpha,r<\infty$, let $0<p<\infty$ and $0<q\leq\infty$, and use
the convention $1/\infty=0$.  If
\begin{equation}\label{eq:sharp-failure}
 \frac1p>
 \frac{\alpha}{r}+\min\!\left\{1,\frac1q\right\},
\end{equation}
then there exist $f\in C^1([-1,1])$ satisfying
$f(0)=0$ and $L_\alpha(f)<\infty$, and
$\lambda\in\ell^r$ satisfying $\|\lambda\|_\infty\leq1$, for which
\[
 S_{\Psi_{f,\lambda}}:\Sp q\longrightarrow\Sp p
\]
is unbounded.  In particular, if $0<q<1$ and
$1/p=\alpha/r+1/q$, then this multiplier need not be bounded.
\end{theorem}

\begin{proof}
Put  $$s=\frac{r}{\alpha}. $$

Suppose first that $0<q\leq1$ (by 
\eqref{eq:sharp-failure}, we have  $p<1$).  Define
\[
 t=\frac{p}{1-p},
 \qquad
 \frac1t=\frac1p-1>\frac1s  
\]
(so that $0<t<s$), and choose a strictly decreasing positive sequence
$d\in\ell^s\setminus\ell^t$ with $\|d\|_\infty\leq1$. By Lemma~\ref{lem:diagonal-realization}, there exist $f$ and $\lambda$ such that
\begin{equation}\label{eq:diagonal-symbol}
 \Psi_{f,\lambda}(j,k)=d_j\delta_{jk}.
\end{equation}
Let $x=(x_j)$ be finitely supported, with $x_j\geq0$ and
$\sum_jx_j=1$.  Put $u_j=\sqrt{x_j}$ and
$A_x=u\otimes u$, where
$(u\otimes u)\xi=\langle\xi,u\rangle u$.  The operator $A_x$ has rank $1$ and has
one singular value equal to $1$, so
\[
 \|A_x\|_q=1 .
\]
Since $(A_x)_{jk}=\sqrt{x_jx_k}$, formula
\eqref{eq:diagonal-symbol} gives
\[
\begin{aligned}
 \bigl(S_{\Psi_{f,\lambda}}(A_x)\bigr)_{jk}
 =
 \Psi_{f,\lambda}(j,k)(A_x)_{jk}=d_j\delta_{jk}\sqrt{x_jx_k}
 =d_jx_j\delta_{jk}.
\end{aligned}
\]
Consequently,
\[
 S_{\Psi_{f,\lambda}}(A_x)
 =\operatorname{diag}(d_jx_j)_{j\geq1},
\]
and hence
\begin{equation}\label{eq:rank-one-output}
 \bigl\|S_{\Psi_{f,\lambda}}(A_x)\bigr\|_p
 =
 \left(\sum_jd_j^px_j^p\right)^{1/p}.
\end{equation}
For $N\geq1$, apply \eqref{eq:rank-one-output} to the particular
vector $x=x^{(N)}$ defined by
\[
 D_N=\sum_{j=1}^N d_j^t,
 \qquad
 x_j^{(N)}=
 \begin{cases}
  d_j^t/D_N,&1\leq j\leq N,\\
  0,&j>N.
 \end{cases}
\]
In particular, $\sum_jx_j^{(N)}=1$. 
Since $p+tp=t$, it follows from  
\eqref{eq:rank-one-output} that 
\[
\begin{aligned}
 \bigl\|S_{\Psi_{f,\lambda}}(A_{x^{(N)}})\bigr\|_p
 =
 \left(
  \sum_{j=1}^N
  d_j^p\left(\frac{d_j^t}{D_N}\right)^p
 \right)^{1/p}=D_N^{-1}
   \left(\sum_{j=1}^N d_j^{p+tp}\right)^{\frac{1}{p}}
 =D_N^{\frac{1}{p}-1}
 =D_N^{\frac{1}{t}}.
\end{aligned}
\]
This tends to infinity as $N\to\infty$ because
$d\notin\ell^t$.  Thus the multiplier is unbounded.

Now suppose that $1<q\leq\infty$.  Condition
\eqref{eq:sharp-failure} implies $p<q$, so we may define $t>0$ by
\[
 \frac1t=\frac1p-\frac1q.
\]
Condition \eqref{eq:sharp-failure} gives $1/t>\frac{\alpha}{r}=1/s$.  Choose a
strictly decreasing positive sequence
$d\in\ell^s\setminus\ell^t$ with $\|d\|_\infty\leq1$.  
By
Lemma~\ref{lem:diagonal-realization}, there exist $f$ and $\lambda$
such that
\[
 \Psi_{f,\lambda}(j,k)=d_j\delta_{jk}.
\]
For every finitely supported scalar sequence $y=(y_j)_{j\geq1}$, put
\[
 B_y:=\operatorname{diag}(y_j)_{j\geq1}.
\]
Then $B_y$ has finite rank (also when $q=\infty$), and, with
$dy:=(d_jy_j)_{j\geq1}$,
\[
 S_{\Psi_{f,\lambda}}(B_y)=B_{dy},\qquad
 \|B_y\|_q=\|y\|_{\ell^q},\qquad
 \bigl\|S_{\Psi_{f,\lambda}}(B_y)\bigr\|_p
 =\|dy\|_{\ell^p}.
\]
If $q<\infty$, then, for $N\geq1$, put
\[
 D_N=\sum_{j=1}^N d_j^t,
 \qquad
 y_j^{(N)}=
 \begin{cases}
  d_j^{t/q}/D_N^{1/q},&1\leq j\leq N,\\
  0,&j>N.
 \end{cases}
\]
We have $\|B_{y^{(N)}}\|_q=\|y^{(N)}\|_{\ell^q}=1$ and,
since $p+pt/q=t$, it follows that
\[
 {\bigl\|S_{\Psi_{f,\lambda}}(B_{y^{(N)}})\bigr\|_p
 =\|dy^{(N)}\|_{\ell^p}}
 =
 D_N^{-1/q}
 \left(\sum_{j=1}^N d_j^t\right)^{1/p}
 =D_N^{1/t}\longrightarrow\infty
 \qquad(N\to\infty).
\]
If $q=\infty$, then $t=p$; taking $y_j^{(N)}=1$ for
$1\leq j\leq N$ and $y_j^{(N)}=0$ for $j>N$ gives
\[
 \begin{aligned}
  \|B_{y^{(N)}}\|_\infty
  &=\|y^{(N)}\|_{\ell^\infty}=1,\\
  \bigl\|S_{\Psi_{f,\lambda}}(B_{y^{(N)}})\bigr\|_p
  &=\|dy^{(N)}\|_{\ell^p}
    =\left(\sum_{j=1}^N d_j^p\right)^{1/p}
    \longrightarrow\infty
    \qquad(N\to\infty).
 \end{aligned}
\]
This proves the theorem.
\end{proof}

\begin{proof}[Proof of Theorem~\ref{thm:main-classification}]
Assume \eqref{eq:full-condition}.  If $q=\infty$, then Corollary
\ref{cor:qinf-positive-range} applies.  If $q<\infty$ and
$p<\infty$,  then 
we appeal to Corollary~\ref{cor:finite-positive}, which concludes the proof in this case.  Finally, if
$q<\infty$ and $p=\infty$, choose a sufficiently large finite
$p_0$ so that
\[
 \frac1{p_0}\leq
 \frac{\alpha}{r}+\min\!\left\{1,\frac1q\right\}.
\]
Again appealing to Corollary~\ref{cor:finite-positive}, we obtain that   $S_{\Psi_{f,\lambda }}$ maps
$\Sp q$ into $\Sp{p_0}\subseteq\K=\Sp\infty$.

Conversely, if \eqref{eq:full-condition} fails, then $p<\infty$ and
Theorem~\ref{thm:sharp-failure} supplies an unbounded multiplier.
\end{proof}

\begingroup

\section{Final remarks}
\label{sec:semifinite-doi}

Let $(\mathcal M,\tau)$ be a   semifinite von
Neumann algebra equipped with a normal faithful semifinite trace.
For $0<p<\infty$, we denote by $L^p(\mathcal M,\tau)$ the associated
noncommutative $L^p$-space and put
$L^\infty(\mathcal M,\tau)=\mathcal M$.  If $X$ is
$\tau$-measurable and $\mu_t(X)$ is its generalized singular-value
function, then
\[
 \|X\|_p=\left(\int_0^\infty\mu_t(X)^p\,dt\right)^{1/p},
\]
We refer to \cite{FackKosaki,LSZ,DPS} for this notation.

For $f\in C^1([-1,1])$, define
\[
 f^{[1]}(s,t)=
 \begin{cases}
 \dfrac{f(s)-f(t)}{s-t},&s\ne t,\\[1ex]
 f'(s),&s=t.
 \end{cases}
 \qquad
 L_\alpha(f):=\sup_{0<|t|\leq1}
 \frac{|f'(t)|}{|t|^\alpha}.
\]
If $A,B\in\mathcal M$ are self-adjoint contractions with spectral
measures $E_A,E_B$, define the double operator integral (DOI)
\begin{equation}\label{def:semifinite-doi}
 T_{f^{[1]}}^{A,B}(X)
 :=
 \int_{[-1,1]}\!\int_{[-1,1]}
 f^{[1]}(s,t)\,dE_A(s)X\,dE_B(t),
\end{equation}
initially on $L^2(\mathcal M,\tau)$\cite{PS09,PS10,ST,PSW}.  The following theorem extends
the main result of this paper to semifinite noncommutative
$L^p$-spaces.

\begin{theorem} \label{thm:semifinite-doi}
Let $0<\alpha<\infty$, let $f\in C^1([-1,1])$, and suppose
$L_\alpha(f)<\infty$.

\begin{enumerate}[label=\textup{(\roman*)}]
\item Let $0<r<\infty$, $1\leq q\leq\infty$, and define $p$ by
\begin{equation}\label{eq:doi-scaling}
 \frac1p=\frac1q+\frac\alpha r,
 \qquad \frac1\infty:=0.
\end{equation}
If $A=A^*$ and $B=B^*$ belong to
$L^r(\mathcal M,\tau)\cap\mathcal M$ and
$\|A\|_\infty,\|B\|_\infty\leq1$, then
$T_{f^{[1]}}^{A,B}$ has a bounded extension
\[
 T_{f^{[1]}}^{A,B}:L^q(\mathcal M,\tau)
 \longrightarrow L^p(\mathcal M,\tau),
\]
and
\begin{equation}\label{eq:semifinite-doi-main}
 \|T_{f^{[1]}}^{A,B}(X)\|_p
 \leq C_{\alpha,p,q}L_\alpha(f)
 \bigl(\|A\|_r^r+\|B\|_r^r\bigr)^{\alpha/r}\|X\|_q.
\end{equation}
For $q=\infty$, the extension is understood as the canonical
$L^p$-limit of its finite spectral truncations.

\item The following limiting case is included for comparison.  Let
$r=\infty$ and $1<q<\infty$.  With the convention
$\alpha/\infty:=0$, \eqref{eq:doi-scaling} yields $p=q$.  For arbitrary self-adjoint
contractions $A,B\in\mathcal M$, we have 
\begin{equation}\label{eq:semifinite-rinfty}
 \|T_{f^{[1]}}^{A,B}(X)\|_q
 \leq C_qL_\alpha(f)\|X\|_q.
\end{equation}
\end{enumerate}
\end{theorem}

Part~\textup{(ii)} is not a new result.  The derivative condition
implies that $f$ is Lipschitz on $[-1,1]$, with Lipschitz constant at
most $L_\alpha(f)$.  After extending $f$ to a Lipschitz function on
$\mathbb R$, the zero-diagonal part of
\eqref{eq:semifinite-rinfty} is precisely the strong $(q,q)$ DOI
estimate in 
\cite[Theorem~7]{PotapovSukochev2011}.

Part~\textup{(i)} follows by carrying the  argument for
the Schur multiplier in Section~\ref{s 3} over to semifinite double operator
integrals: matrix blocks and their ranks are replaced by spectral
rectangles and their traces.  On adjacent shells one uses
\cite[Theorem~1.2]{CPSZ} for $q=1$ and
\cite[Theorem~7]{PotapovSukochev2011} for $1<q<\infty$.
The remaining
block estimates and summations are unchanged, so we omit the
details.

\begin{remark}\label{rem:semifinite-scope}
The theorem covers every $1\leq q\leq\infty$ when
$0<r<\infty$.  At the limiting value $r=\infty$,  we have  $p=q$.  
The two cases
$(p,q)=(1,1)$ and $(p,q)=(\infty,\infty)$ are not included: under
the present $C^1$ hypothesis, strong boundedness at these two
endpoints fails in general (see e.g. 
\cite[Lemma~6.1 and Theorem~6.7]{ShulmanTodorovTurowska}, \cite{Kato,Davies,Far},  \cite[Introduction]{PotapovSukochev2011}).
If one strengthens the
assumption to $f'\in\operatorname{Lip}_\beta([-1,1])$ for some
$\beta>0$, then both endpoints follow from
\cite[Theorem~5]{BirmanSolomyak1965}.


\end{remark}

\begin{remark}
The restriction $q\geq1$ is essential in the general semifinite
setting, which is different from Theorem \ref{thm:main-classification}.  For $0<q<1$, one has the inclusion
$\mathcal S^q\subset\mathcal S^1$, which fails for atomless
semifinite noncommutative $L^q$-spaces.
Thus the result for
$0<q<1$ is specific to Schatten ideals and does not extend to arbitrary
semifinite von Neumann algebras.
Indeed, suppose that the estimate
\begin{equation}\label{eq:hypothetical-q-less-one}
 \bigl\|T_{f^{[1]}}^{A,B}(X)\bigr\|_p
 \leq
 C_{\alpha,p,q}L_\alpha(f)
 \bigl(\|A\|_r^r+\|B\|_r^r\bigr)^{\alpha/r}
 \|X\|_q
\end{equation}
held for all semifinite von Neumann algebras.  Let
 $\mathcal M=L^\infty(0,1)$
with the Lebesgue trace, and choose a measurable set
$E\subset(0,1)$ with measure $\varepsilon\in(0,1)$.  Put
\[
 e=\mathbf 1_E,\qquad
 A=B=ae,\qquad X=e, \qquad  f(t)=\frac{t|t|^\alpha}{\alpha+1}, 
\]
where $0<a\leq1$. 
Then $f'(t)=|t|^\alpha$, so that $L_\alpha(f)=1$.  Observe that 
\(
 T_{f^{[1]}}^{A,A}(e)
 =f^{[1]}(a,a)e
 =f'(a)e
 =a^\alpha e.
\)
Consequently,
\begin{equation}\label{eq:diffuse-left-hand-side}
 \bigl\|T_{f^{[1]}}^{A,A}(e)\bigr\|_p
 =a^\alpha\|e\|_p
 =a^\alpha\varepsilon^{1/p}.
\end{equation}
Moreover,  since
\( 
 \|A\|_r^r=\tau(|A|^r)=a^r\varepsilon=
 \|B\|_r^r,
\)
it follows that 
\begin{equation}\label{eq:diffuse-spectral-factor}
 \bigl(\|A\|_r^r+\|B\|_r^r\bigr)^{\alpha/r}
 =
 (2a^r\varepsilon)^{\alpha/r}
 =
 2^{\alpha/r}a^\alpha\varepsilon^{\alpha/r}.
\end{equation}
We also have
\begin{equation}\label{eq:diffuse-input-norm}
 \|X\|_q=\|e\|_q
 =\tau(e)^{1/q}
 =\varepsilon^{1/q}.
\end{equation}
Substituting
\eqref{eq:diffuse-left-hand-side}--\eqref{eq:diffuse-input-norm}
into \eqref{eq:hypothetical-q-less-one}, we obtain
\[
 a^\alpha\varepsilon^{1/p}
 \leq
 C_{\alpha,p,q}
 2^{\alpha/r}a^\alpha
 \varepsilon^{\alpha/r+1/q}.
\]
If 
\(
 \frac1p=1+\frac{\alpha}{r},
\)
then we have 
\[
 \varepsilon^{\,1-1/q}
 \leq C_{\alpha,p,q}2^{\alpha/r}.
\]
For $0<q<1$, one has  
\( 
 \varepsilon^{\,1-1/q}
 \longrightarrow\infty
~\text{as }\varepsilon\downarrow0,
\)
which is a contradiction.  Hence, \eqref{eq:hypothetical-q-less-one} cannot hold
uniformly over all semifinite von Neumann algebras when
$0<q<1$ and $1/p=1+\alpha/r$.
\end{remark}



The DOI formulation also gives the usual perturbation and
quasicommutator consequences.

\begin{corollary}\label{cor:semifinite-quasicommutator}
Under the assumptions of part~\textup{(i)} of
Theorem~\ref{thm:semifinite-doi}, if
$Z\in\mathcal M$ and
$AZ-ZB\in L^q(\mathcal M,\tau)\cap L^2(\mathcal M,\tau)$, then
\[
 \|f(A)Z-Zf(B)\|_p
 \leq C_{\alpha,p,q}L_\alpha(f)
 \bigl(\|A\|_r^r+\|B\|_r^r\bigr)^{\alpha/r}
 \|AZ-ZB\|_q.
\]
In particular, if
$A-B\in L^q(\mathcal M,\tau)\cap L^2(\mathcal M,\tau)$, then
\[
 \|f(A)-f(B)\|_p
 \leq C_{\alpha,p,q}L_\alpha(f)
 \bigl(\|A\|_r^r+\|B\|_r^r\bigr)^{\alpha/r}
 \|A-B\|_q.
\]
\end{corollary}

\begin{proof}
Put $D=AZ-ZB$.  Since $D\in L^2$, the spectral calculus gives
\[
 \begin{aligned}
 T_{f^{[1]}}^{A,B}(D)
 &=\iint f^{[1]}(s,t)(s-t)\,dE_A(s)Z\,dE_B(t)\\
 &=\iint\bigl(f(s)-f(t)\bigr)\,dE_A(s)Z\,dE_B(t)\\
 &=f(A)Z-Zf(B).
 \end{aligned}
\]
Theorem~\ref{thm:semifinite-doi}
proves the first estimate.  Taking $Z=1$ delivers the second one.
\end{proof}

\begin{remark}
\label{rem:comparison-hsz}
Let $0<\theta<1$.  It is shown in 
\cite[Theorem~1.2]{HSZ} (see also \cite[Theorem 3.1]{HNSZ})  that, for every $0<u<\infty$ and an
appropriate $d=d(u)$, every function $g$ in their Sobolev class
$S_{d,\theta}$ satisfies
\[
 \|g(A)-g(B)\|_u
 \leq C_{u,\theta}\|g\|_{S_{d,\theta}}
       \bigl\||A-B|^\theta\bigr\|_u
 =C_{u,\theta}\|g\|_{S_{d,\theta}}
       \|A-B\|_{\theta u}^{\theta}.
\]
In particular, $\theta u <u$.  
The result in 
Corollary~\ref{cor:semifinite-quasicommutator}
is different from the above inequality. Indeed, 
 in that corollary, we have 
  $p<q$ and the inequality depends on the $L_r$-norms of $A$ and $B$, which is not the case for the inequality above.

\end{remark}

\endgroup

{\bf Acknowledgement} The authors would like to thank Anna Tomskova for helpful comments. 


\end{document}